\documentclass{article}

\usepackage[english]{babel}

\usepackage[letterpaper,top=2cm,bottom=2cm,left=3cm,right=3cm,marginparwidth=1.75cm]{geometry}

\usepackage{amsmath}
\usepackage{amsfonts}
\usepackage{amsthm}
\usepackage{graphicx}
\usepackage[colorlinks=true, allcolors=blue]{hyperref}
\usepackage{authblk}
\usepackage{xcolor}

\usepackage{algorithm,algpseudocode}

\usepackage{mathrsfs}

\newtheorem{theorem}{Theorem}
\newtheorem{lemma}[theorem]{Lemma}
\newtheorem{remark}[theorem]{Remark}
\newtheorem{corollary}[theorem]{Corollary}

\newtheorem{definition}[theorem]{Definition}

\newtheorem{example}[theorem]{Example}

\newcommand{\N}{\mathbb{N}}
\newcommand{\R}{\mathbb{R}}
\newcommand{\Z}{\mathbb{Z}}

\newcommand{\rev}[1]{{\color{black} #1}}

\DeclareMathOperator*{\diameter}{diam}
\DeclareMathOperator*{\closure}{cl}

\newcommand{\diam}[1]{\diameter\left( #1\right)}
\newcommand{\cl}[1]{\closure\left( #1\right)}

\begin{document}
\title{Bounding the Optimal Number of Policies for Robust $K$-Adaptability}
\author[1]{Jannis Kurtz\footnote{j.kurtz@uva.nl}}
\affil[1]{Amsterdam Business School, University of Amsterdam, 1018 TV Amsterdam, Netherlands }

\date{}
\maketitle

\begin{abstract}
In the realm of robust optimization the $k$-adaptability approach is one promising method to derive approximate solutions for two-stage robust optimization problems. Instead of allowing all possible second-stage decisions, the $k$-adaptability approach aims at calculating a limited set of $k$ such decisions already in the first-stage before the uncertainty is revealed. The parameter $k$ can be adjusted to control the quality of the approximation. However, not much is known on how many solutions $k$ are needed to achieve an optimal solution for the two-stage robust problem. In this work we derive bounds on $k$ which guarantee optimality for general non-linear problems with integer decisions where the uncertainty appears in the objective function or in the constraints. For convex uncertainty sets we show that for objective uncertainty the bound depends linearly on the dimension of the uncertainty, while for constraint uncertainty the dependence can be exponential, still providing the first generic bound for a wide class of problems. \rev{Additionally, we provide approximation guarantees if $k$ is smaller than the derived bounds.} The results give new insights on how many solutions are needed for problems as the decision dependent information discovery problem or the capital budgeting problem with constraint uncertainty. Finally, for finite uncertainty sets we show that calculating the minimal $k$ for which $k$-adaptable and two-stage problems are equivalent is NP-hard and derive a greedy method which approximates this $k$ for the case where no first-stage decisions exist. 
\end{abstract}

\section{Introduction}
Two-stage robust optimization problems appear in a variety of applications where decisions are influenced by uncertain parameters, e.g., the demands of a customer, the travel time, or the population density of a certain district; see \cite{gorissen2015practical,yanikouglu2019survey}. As common in robust optimization, it is assumed that the uncertain parameters lie in an uncertainty set which is pre-constructed by the user. In the two-stage robust setting some of the decisions have to be taken \textit{here-and-now} while some decisions can be taken after the uncertain parameters of the problem are known (\textit{wait-and-see decisions}). The goal is to find a here-and-now decision which optimizes the worst possible objective value over all scenarios in the uncertainty set.

While a large amount of works concentrate on the case where the decision variables are continuous, many real-world applications and combinatorial problem structures require integer decisions; \cite{buchheim2018robust}. Unfortunately, two-stage robust optimization problems with integer wait-and-see decisions are computationally extremely challenging while at the same time the variety of solution methods is still limited. For the case where the uncertain parameters only appear in the objective function promising algorithms based on column-generation or branch \& bound methods were developed (\cite{kammerling2020oracle,arslan2022decomposition,detienne2024adjustable}). At the same time the constraint uncertainty case is still insufficiently investigated. Here, classical column-and-constraint generation (CCG) approaches were adapted for the general mixed-integer case \cite{zhao2012exact} or for interdiction-type problems \cite{lefebvre2023using}. Recently, a neural network supported CCG was developed which can calculate heuristic solutions of high quality much faster than state-of-the-art approaches \cite{dumouchelle2024neurro}. 

One promising method to approximate two-stage robust optimization problems is the $k$-adaptability approach, where, instead of considering all wait-and-see solutions, a limited set of $k$ such solutions is calculated in the first-stage such that the best of it can be chosen after the uncertain parameters are known. This approach was first studied in \cite{bertsimas2010finite} and gained more attention later in \cite{hanasusanto2015k,subramanyam2020k,kurtz2024approximation,romeijnders2021piecewise}. A related special case of the problem, where no first-stage solutions are considered, sometimes called \textit{min-max-min robust optimization}, was studied first in \cite{buchheim2017min} and later in several other works \cite{chassein2019faster,buchheim2018complexity,goerigk2020min,chassein2021complexity,arslan2022min}. 

One important research question is: How many second-stage solutions $k$ do we need such that the $k$-adaptability approach returns an optimal (or approximately optimal) solution of the two-stage robust optimization problem (2RO)? If we know such a number $k$ we can use the $k$-adaptability approach (k-ARO) to solve the two-stage robust problem. Furthermore, it provides insights on the complexity of the uncertainty set in connection with the second-stage problem, since larger values for $k$ indicate a more diverse set of scenarios and required second-stage reactions. \rev{Finally, in some applications as radiotherapy or disaster management a small number of solutions is desired since each solution has to be approved by a doctor or planner; see \cite{qiu2025kadaptabilityapproachprotonradiation,weller2025streamlining}. In this case we want to know if a small number of solutions exists which already provide the best possible objective value or a certain approximation guarantee}. However, insights on the number of wait-and-see solutions needed for \rev{(approximate)} optimality are sparse.

\rev{
In \cite{hanasusanto2015k} the authors show that if the uncertainty only appears in the objective function and if this objective function is linear, at most $k=n+1$ second-stage policies are needed, where $n$ is the minimum of the dimension of the uncertain parameters and the dimension of the second-stage. This coincides with the result observed in \cite{buchheim2017min} for min-max-min robust combinatorial optimization problems, which is a special case of the $k$-adaptability problem. In \cite{kurtz2024approximation} it was shown that under objective uncertainty, if we want to approximate \eqref{eq:2StageRO} by a factor of $1+\alpha(n_y)$, where $n_y$ is the dimension of the second-stage problem, then it is enough to use $k=qn_y$ policies where $q=\frac{M(n_y)}{M(n_y)+\alpha(n_y)}$ and $M(n_y)$ is a value depending on the problem parameters and the dimension $n_y$. In this work we will generalize the latter results to the case of continuous objective functions.

In case the uncertain parameters appear in the constraints, the best known bounds on $k$ are not very promising. In \cite{bertsimas2010finite} the authors argue that under a given continuity assumption and for continuous second-stage decisions the $k$-adaptability approach converges to an optimal solution of \eqref{eq:2StageRO} for $k\to\infty$. Unfortunately, this result is not correct as it was shown later in \cite{kedad2023continuity}. The authors provide counterexamples where the $k$-adaptability approach does not lead to an optimal solution of \eqref{eq:2StageRO} for any $k\in\mathbb N$. However, the authors show that the continuity assumption can be adjusted such that the original convergence result holds. Again, for the case of continuous second-stage decisions, the authors in \cite{el2018piecewise} derive approximation guarantees which \eqref{eq:k-adaptability} provides for \eqref{eq:2StageRO} and show that, if the number of policies $k$ is bounded by a polynomial in the problem parameters, \eqref{eq:k-adaptability} cannot approximate \eqref{eq:2StageRO} better than by a factor of \rev{$m^{1-\varepsilon}$ for any given $\varepsilon >0$, where $m$ is the number of second-stage constraints}.  

In the setting which is studied in this work, namely the set of second-stage solutions is bounded and only contains integer solutions, the number of possible second-stage policies is finite. Hence, the convergence discussion above is not necessary, since trivially if we set $k$ to the number of possible second-stage solutions  the $k$-adaptability problem will return the optimal solution of \eqref{eq:2StageRO}. It is shown in \cite{hanasusanto2015k} that indeed there are problem instances where all second-stage solutions are needed, hence finding a better bound is impossible in the general setting. However, in this work we will derive better bounds on $k$ for certain problem structures.
}

%
% In \cite{hanasusanto2015k,buchheim2017min} it was shown that $k=n+1$ solutions are enough for linear problems with objective uncertainty (where $n$ is the minimum of the \rev{second-stage dimension} and the dimension of the uncertainty). For the constraint uncertainty case the authors in \cite{hanasusanto2015k} present an example where all second-stage solutions are needed to guarantee optimality. To the best of our knowledge there are no better bounds known for the constraint uncertainty case with integer recourse.

\paragraph*{Contributions}
\begin{itemize}
    \item We show that for convex uncertainty sets in the objective uncertainty case the bound on $k$ which is known for the linear case holds even if we consider general non-linear objective functions which are concave in the uncertain parameters. As a consequence we can show for the first time that for robust optimization with decision dependent information discovery at most $k=n_\xi+1$  solutions are needed to guarantee optimality, where $n_\xi$ is the dimension of the uncertainty. Based on the latter bounds on $k$, we derive bounds on the approximation guarantee of the $k$-adaptability approach for arbitrary values of $k$.
    \item We derive bounds on $k$ to guarantee optimality for the constraint uncertainty case. To this end we introduce a new concept called \textit{recourse-stability} which leads to a bound on $k$ which depends on the uncertainty dimension and the number of recourse-stable regions needed to cover the uncertainty set. We show that for certain problem structures the \rev{derived bound on $k$ is significantly smaller than the previously known bounds and provide an example which shows that the bound is tight. Additionally we provide approximation guarantees for smaller values of $k$}.
    \item For finite uncertainty sets, we show that calculating the minimal $k$ for which $k$-adaptable and two-stage problems are equivalent is NP-hard. We derive a greedy method which approximates this value for the case without first-stage decisions.
\end{itemize}

\section{Preliminaries}

\subsection{Notation and Preliminaries}
For any given positive integer $n$ we denote $[n]=\{ 1,2,\ldots ,n\}$, we denote all $n$-dimensional vectors of non-negative real numbers as $\mathbb R_+^n:=\{ x\in\R^n: x\ge 0\}$ and all $n$-dimensional vectors of non-negative integers as $\mathbb Z_+^n:=\{ x\in\Z^n: x\ge 0\}$. The euclidean norm is denoted as $\|\cdot\|$, i.e., $\| x\|=\sqrt{\sum_{i=1}^{n}x_i^2}$ for any $x\in\R^n$. For a given set $S\subseteq \R^n$ we define the diameter of the set as $\diam{\mathcal S}=\max_{x,y\in \mathcal S}\| x-y\|$, the closure of the set as $\cl{\mathcal S}$, where a point $x\in\R^n$ is contained in the closure of $\mathcal S$ if and only if for every radius $\varepsilon>0$ there exists a point $s\in\mathcal S$ with $\|x-s\|<\varepsilon$.

For any $\mathcal X\subseteq \R^n$ we call a function $f:\mathcal X\to \R$ \textit{convex} if and only if \[f(\lambda x + (1-\lambda) y) \le \lambda f(x) + (1-\lambda) f(y)\] holds for all $x,y\in\mathcal X$ and $0\le \lambda\le 1$. A function $f$ is \textit{concave} if $-f$ is convex. The function $f$ is \textit{Lipschitz continuous} with Lipschitz constant $L>0$ if and only if 
\[
|f(x)-f(y)|\le L\| x-y\|
\]
holds for all $x,y\in\mathcal X$.

One preliminary result we will use in Section \ref{sec:objective_uncertainty} and \ref{sec:constraint_uncertainty} was derived in \cite{calafiore2005uncertain}. In this work the authors study convex optimization problems of the form
\begin{align*}
    \mathcal P: \ \min_{x\in\R^n} \ & c^\top x \\
    s.t. \quad & x\in \mathcal X_i \quad i\in [m] 
\end{align*}
where $m\in \Z_+$, $c\in\R^n$ and $\mathcal X_i$ is a closed and convex set for every $i\in [m]$. \rev{Note that the more general set-based definition of a single constraint $x\in \mathcal X_i$ can involve multiple function-based constraints which form the set $\mathcal X_i$.} The authors define the constraint $\mathcal X_k$ to be a \textit{support constraint} if removing it from the problem leads to a strictly better optimal value compared to the original problem $\mathcal P$. They prove the following theorem.
\begin{theorem}[\cite{calafiore2005uncertain}]\label{thm:califiore}
The number of support constraints for Problem $\mathcal P$ is at most $n$.
\end{theorem}
Note that assuming a linear objective function in $\mathcal P$ is without loss of generality since we can always move a convex objective function into the constraints by using the epigraph reformulation.

\subsection{Problem Definitions}\label{sec:problem_definitions}
In this work we consider the general class of (non-linear) two-stage robust optimization problems of the form
\begin{equation}\label{eq:2StageRO}\tag{2RO}
    \inf_{x\in \mathcal X} \sup_{\xi\in \mathcal U} \inf_{y\in \mathcal Y(x)} \ f(x,y,\xi)
\end{equation}
where $\mathcal X\subseteq \mathbb R^{n_x}$ is an arbitrary compact set containing all possible first-stage decisions, $\mathcal Y(x)\subseteq \mathcal Y \subset \mathbb Z^{n_y}$ is the set of feasible second-stage decisions $y$ which can depend on the chosen first-stage decision $x$ and $\mathcal U\subset \mathbb R^{n_\xi}$ is a compact uncertainty set containing all possible scenarios $\xi$. In Sections \ref{sec:objective_uncertainty} and \ref{sec:constraint_uncertainty} we study the case where $\mathcal U$ is convex, while in Section \ref{sec:finite_uncertainty} we study finite uncertainty sets $\mathcal U$. We assume that $\mathcal Y$ is bounded, i.e., it contains a finite number of solutions. Furthermore, $f: \mathcal X\times \mathcal Y \times \mathcal U \to \mathbb R$ is an arbitrary function, if not stated otherwise.

While the uncertainty parameters $\xi$ seem to appear only in the objective function in \eqref{eq:2StageRO}, the problem definition also covers the case of constraint uncertainty due to the generality of the objective function $f$. Indeed, we will consider the case of constraint uncertainty in Section \ref{sec:constraint_uncertainty}, by considering the function 
\[
f(x,y,\xi):=\begin{cases}
    g(x,y,\xi) & \text{ if } A(\xi)x + B(\xi)y \ge h(\xi) \\
    \infty & \text{ otherwise,}
\end{cases}
\]
where $g: \mathcal X\times \mathcal Y \times \mathcal U \to \mathbb R$ is a given continuous objective function, $A(\xi)\in \mathbb R^{m\times n_x}$, $B(\xi)\in \mathbb R^{m\times n_y}$ and $h(\xi)\in \mathbb R^{m}$ are constraint parameters which are given as functions of the uncertain parameters. \rev{The latter function $f$ ensures that any $x\in\mathcal X$ for which a $\xi\in\mathcal U$ exists such that the second-stage constraint system is infeasible has objective value $\infty$ in Problem \eqref{eq:2StageRO}.} 
% We do not assume that relatively complete recourse holds, i.e., there can exist an $x\in \mathcal X$ and $\xi\in \mathcal U$ such that no feasible second-stage solution exists.

The $k$-adaptability approach aims at finding approximate solutions $x\in \mathcal X$ for \eqref{eq:2StageRO}. The idea is, for a fixed parameter $k\in\mathbb N$, to calculate a set of $k$ second-stage policies $y^1,\ldots ,y^k$ already in the first stage, and choose the best of it in the second-stage after the scenario is known. This leads to the problem
\begin{equation}\label{eq:k-adaptability}\tag{k-ARO}
    \text{opt}(k):=\inf_{\substack{x\in \mathcal X \\ y^1, \ldots , y^k\in \mathcal Y(x)}} \sup_{\xi\in \mathcal U} \inf_{i=1,\ldots ,k} \ f(x,y^i,\xi).
\end{equation}
Using this idea we cannot guarantee that the calculated solution $x\in\mathcal X$ is optimal for \eqref{eq:2StageRO}. In fact the quality of the optimal $k$-adaptable solution depends on the parameter $k$. The larger $k$, the better is the approximation for the original two-stage problem $\eqref{eq:2StageRO}$. On the other hand, the larger we choose $k$, the more numerically challenging Problem $\eqref{eq:k-adaptability}$ becomes, since we have to introduce more second-stage decision variables.

%Hence, an interesting research question is: How many second-stage policies $k$ do we need, such that the optimal solution of \eqref{eq:k-adaptability} is also optimal for \eqref{eq:2StageRO}?

\rev{
In the following we provide problem formulations of several optimization problems which are later used in this paper to verify the results.

\paragraph*{Robust Optimization with Decision-Dependent Information Discovery}
This problem was introduced in \cite{vayanos2020robust} and later studied in \cite{paradiso2022exact,omer2023combinatorial}. In the first-stage we can invest into discovering the values of a subset of uncertain parameters where afterwards the worst-case is taken with respect to the information which was discovered. In both of the works \cite{vayanos2020robust,paradiso2022exact} the $k$-adaptability version of the problem is studied which is given as
\[
\min_{\substack{w\in\mathcal W \\ y^1,\ldots ,y^k\in\mathcal Y}} \max_{\bar \xi \in \mathcal U} \min_{i=1,\ldots ,k} \max_{\xi \in \mathcal U(w,\bar \xi)} \xi^\top Cw + \xi^\top P y^i
\]
for matrices $C,P$ of appropriate size, where $\mathcal W\subseteq \{ 0,1\}^{n_w}$, $\mathcal Y\subseteq \{ 0,1\}^{n_y}$, $\mathcal U\subset \mathbb R^{n_\xi}$ is a polyhedral uncertainty set and $U(w,\bar \xi) = \{ \xi\in \mathcal U: w_i\xi_i = w_i\bar \xi_i, i=1,\ldots ,n_\xi\}$.
\paragraph*{Capital Budgeting}
The $k$-adaptable version of the two-stage robust capital budgeting problem (CB) with objective uncertainty was studied in \cite{subramanyam2020k}. The problem is given as
\[
\max_{\substack{x\in \mathcal X \\ y^1, \ldots ,y^k\in\mathcal Y(x,\xi)}} \min_{\xi\in \mathcal U}\max_{i=1,\ldots , k} r(\xi)^\top (x+\kappa y^i)
\]
where $\mathcal X=\{ 0,1\}^n$ and $\mathcal Y(x,\xi)=\{ y\in \{ 0,1\}^n: c(\xi)^\top (x+y)\le B, \ x+y \le e\}$ and $c_i(\xi) = (1+\frac{1}{2}\Phi_i^\top \xi)c_i^0$ are the uncertain costs of project $i$ and $\Phi_i$ is the $i$-th row of a given factor loading matrix $\Phi$. Furthermore, $\mathcal U = [-1,1]^\rho$ is an uncertainty set of all realizations of $\rho$ different risk factors and $e$ is the all-one vector. The risk of project $i$ is given as $r_i(\xi) = (1+\frac{1}{2}\Psi_i^\top \xi)r_i^0$ where $\Psi_i$ is the $i$-th row of a given factor loading matrix $\Psi$. The parameter $\kappa$ models the penalty of postponement. Note that the number of risk factors $\rho$ is usually a small number which is independent of $n$.
\paragraph*{Capacitated Facility Location Problem}
We consider a capacitated facility location problem with uncertain transportation costs and uncertain capacities. We have a set of customers $\mathcal J$ and a set of locations $\mathcal I$ and transportation costs $t_{ij}(\xi)$ between each $i\in\mathcal I$ and $j\in\mathcal J$ which depend on the uncertain scenario $\xi\in \mathcal U$. Furthermore, each location $i\in\mathcal I$ has a capacity $C(\xi) = c^\top \xi$ where all entries of $c$ are integer. Each customer $j\in\mathcal J$ has a known integer demand $d_j>0$. We are allowed to open at most $p$ facilities in the first-stage and for every scenario $\xi$ we afterwards have to assign each customer to an opened facilities such that the sum of assigned demands for each facility does not exceed the capacity. The problem can be formulated as a two-stage robust problem with $\mathcal X=\left\{ x\in \{ 0,1\}^{\mathcal I} : \sum_{i\in\mathcal I} x_i \le p\right\}$ and 
\[
\mathcal Y(x,\xi) = \left\{ y\in \{ 0,1\}^{\mathcal I \times \mathcal J}: \ \sum_{i\in\mathcal I} y_{ij} = 1 \ \forall j\in\mathcal J, \ \sum_{j\in\mathcal J} d_j y_{ij} \le c^\top \xi x_i \ \forall i\in\mathcal I\right\}.
\]
The objective function is $g(y,\xi) = \sum_{i\in\mathcal I}\sum_{j\in\mathcal J} t_{ij}(\xi) y_{ij}$.
}

\section{Objective Uncertainty}\label{sec:objective_uncertainty}
In this section we assume that $\mathcal U$ is a convex uncertainty set. In case of objective uncertainty and a linear objective function the $k$-adaptability problem provides an optimal solution of \eqref{eq:2StageRO} if $k\ge \min\{n_y,n_\xi\}+1$; see \cite{hanasusanto2015k}. The following theorem shows that a similar result holds for arbitrary continuous objective functions which are concave in the uncertain parameters. 

\begin{theorem}\label{thm:n+1_policies}
\rev{Let $\mathcal U$ be convex and} $f:\mathcal X\times \mathcal Y\times \mathcal U\to \mathbb R$ a continuous function such that $f(x, y, \xi)$ is concave in $\xi$ for every $x\in \mathcal X, y\in\mathcal Y$ and let $k\ge n_\xi +1$. Then, a solution $x\in\mathcal X$ is optimal for \eqref{eq:k-adaptability} if and only if it is optimal for \eqref{eq:2StageRO}.
\end{theorem}
\begin{proof}
First, note that since all sets $\mathcal X,\mathcal U, \mathcal Y$ are compact and $f$ is continuous, all the maxima and minima in the problem definition \eqref{eq:2StageRO} exist and are finite.

Since $\mathcal Y$ is bounded and contains only integer decisions, we know that for $k=|\mathcal Y|$ the problems \eqref{eq:k-adaptability} and \eqref{eq:2StageRO} are equivalent. Fix any  first-stage decision $x\in\mathcal X$. By using an epigraph reformulation we can rewrite the inner max-min problem of \eqref{eq:k-adaptability} with $k=|\mathcal Y(x)|$ as
\begin{align*}
    \max_{z,\xi} & \ z \\
    s.t. \quad & f(x,y,\xi) - z \ge 0 \quad \forall y\in \mathcal Y(x)\\
    & \xi \in \mathcal U \\
    & z \in \mathbb R.
\end{align*}
Since $f$ is concave and continuous in $\xi$ the function $f(x,y,\xi)-z$ is concave and continuous in $(\xi,z)$. Hence, the latter problem is convex, where for every $y\in \mathcal Y(x)$ the feasible set corresponding to the constraint is closed and convex. Additionally, $\xi \in \mathcal U$ and $z \in \mathbb R$ are convex constraints with closed and convex region. From Theorem \ref{thm:califiore} it follows, that the number of support constraints is at most the dimension of the problem, i.e., $n_\xi + 1$. Hence, we can remove all constraints except $n_\xi + 1$ from the problem without changing the optimal solution. We can conclude that at most $n_\xi + 1$ of the second-stage solutions $y\in \mathcal Y(x)$ are needed. This holds for any $x\in\mathcal X$ which proves the result.
\end{proof}
\rev{The latter bound on $k$ does only depends on the dimension of the uncertainty set and not on the dimension of the decision variables $x$ and $y$.} Furthermore, we do not make any assumptions on the function $f$ regarding $x$ and $y$; especially no convexity in $x$ or $y$ is required. 

\rev{
\begin{remark}
If $n_\xi\le n_y$, Theorem \ref{thm:n+1_policies} coincides with the bound derived in \cite{hanasusanto2015k} for the case that $f$ is a linear function. Note that in \cite{hanasusanto2015k} a transformation $Q \xi$ is considered which leads to a bound which involves the rank of matrix $Q$. However, any such linear transformation can be used in Theorem \ref{thm:n+1_policies} as well by using the objective function $\tilde f(x,y,\xi):=f(x,y,Q \xi)$ which remains concave in the uncertain parameters. Additionally, we use the uncertainty set $$\tilde{\mathcal U}=\left\{ \xi\in\mathbb R^{\text{rank}(Q)}: \sum_{i=1}^{\text{rank}(Q)} Q_{j_i} \xi_i \in \mathcal U\right\}$$ where $Q_{j_1},\ldots ,Q_{j_{\text{rank}(Q)}}$ are linearly independent columns of $Q$. Then, $\tilde{\mathcal U}$ is convex which reduces the dimension of the uncertainty parameters to $n_\xi = \text{rank}(Q)$.
\end{remark}
}

\rev{The following example shows, that the bound derived in Theorem \ref{thm:n+1_policies} is tight up to an additive constant of one.
\begin{example}\label{ex:bound_k_tight_objective}
Consider the following $k$-adaptability problem without first-stage variables $x$ and $n_\xi = n_y$,
\[
\min_{y^1, \ldots , y^k\in \mathcal Y} \max_{\xi\in \mathcal U} \min_{i=1,\ldots ,k} \xi^\top y^i ,
\]
where $\mathcal U= \left\{ \xi\in \mathbb R^{n_\xi}: \sum_{i=1}^{n} \xi_i = 1\right\}$ and $\mathcal Y=\left\{ y\in \{ 0,1\}^{n_y}: \sum y_i = 1\right\}$. First, note that for $k=|\mathcal Y | = n_\xi$ we obtain the best possible value which is given as
\[
\text{opt}(n_\xi) = \max_{\xi\in \mathcal U} \min_{y\in \{ e_1,\ldots ,e_n \}} \xi^\top y = \frac{1}{n},
\]
where $e_i$ is the $i$-th unit vector. Assume now that we are allowed to choose only $k=n_\xi -1$ solutions, then we obtain
\[
\text{opt}(n_\xi-1) = \max_{\xi\in \mathcal U} \min_{y\in \{ e_1,\ldots ,e_n \}\setminus \{e_t\}} \xi^\top y = \frac{1}{n-1} > \text{opt}(n_\xi)
\]
for an arbitrary $t\in [n]$. This shows that we need at least $k=n_\xi$ solutions for optimality.
\end{example}
}

The following example shows that we can apply Theorem \ref{thm:n+1_policies} to the robust optimization problem with decision-dependent information discovery (DDID).

\begin{example}
Consider the Robust Optimization Problem with Decision-Dependent Information Discovery (DDID), defined in Section \ref{sec:problem_definitions}. We can rewrite the problem into the form  \eqref{eq:k-adaptability} where
\[
f(x,y,\bar \xi):= \max_{\xi \in \mathcal U(w,\bar \xi)} \xi^\top Cw + \xi^\top P y.
\]
To apply Theorem \ref{thm:n+1_policies} we have to show that $f$ is concave in $\bar \xi$. We can reformulate $f$ as 
\begin{align*}
    \max_\xi \ & \xi^\top Cw + \xi^\top P y \\
    s.t. \quad & w_i\xi = w_i\bar \xi_i \quad i=1,\ldots ,n_\xi \\
    & \xi \in \mathcal U .
\end{align*}
Taking the dual the problem can be transformed into the minimum of linear functions in $\bar \xi$, which is concave and continuous in $\bar \xi$. Hence, from Theorem \ref{thm:n+1_policies} it follows, that at most $k=n_\xi + 1$ second-stage policies are needed to get an optimal solution for DDID.
\end{example}

The next example shows that for the capital budgeting problem the number of policies needed to guarantee optimality can be very small, namely the number of risk factors plus one.
\begin{example}\label{ex:cb_objective}
Consider the capital budgeting problem (CB) defined in Section \ref{sec:problem_definitions} where no uncertainty appears in the constraints, i.e., the costs $c(\xi)$ are certain. Clearly, the objective function $f(x,y,\xi) = r(\xi)^\top (x+\kappa y)$ is linear (and therefore concave and continuous) in $\xi$ and we can apply Theorem \ref{thm:n+1_policies} to show that at most $k=\rho + 1$ second-stage policies are needed, where $\rho$ is usually a small number of risk-factors. \rev{This result aligns with the results in \cite{hanasusanto2015k} for the case $n_\xi\le n_y$. Note that Theorem \ref{thm:n+1_policies} shows that the same bound on $k$ holds if we replace the objective function in CB with an arbitrary continuous function which is concave in $\xi$.}
\end{example}

Next, we derive approximation bounds which the $k$-adaptability problem provides for  \eqref{eq:2StageRO}.

\begin{theorem}\label{thm:approximation_bound}
\rev{Let $\mathcal U$ be convex and} $f:\mathcal X\times \mathcal Y\times \mathcal U\to \mathbb R$ a continuous function such that $f(x, y, \xi)$ is concave in $\xi$ for every $x\in \mathcal X, y\in\mathcal Y$. Furthermore, assume $f$ is Lipschitz continuous in $y$, i.e., there exists a constant $L>0$ such that
\[
|f(x,y,\xi)-f(x,y',\xi)|\le L \|y-y'\| \quad \forall x\in\mathcal X, \xi\in \mathcal U, y,y'\in \mathcal Y.
\]
Then, for any $s,k\in \mathbb N$ with $s\le k$ it holds
\[
\text{opt}(s)-\text{opt}(k) \le L \text{diam}(\mathcal Y) \rev{\ln\left(\frac{k}{s}\right)}.
\]
\end{theorem}
\begin{proof}
First, we reformulate \eqref{eq:k-adaptability} as
\[
    \min_{\substack{x\in \mathcal X \\ y^1, \ldots , y^k\in \mathcal Y(x)}} \max_{\xi\in \mathcal U} \min_{\substack{\lambda \in \mathbb R_+^k \\ \sum_{i=1}^{k} \lambda_i = 1}} \ \sum_{i=1}^{k} \lambda_if(x,y^i,\xi).
\]
Since $f$ is concave in $\xi$ and $\lambda \ge 0$, also the function $\sum_{i=1}^{k} \lambda_if(x,y^i,\xi)$ is concave in $\xi$. We can apply the classical minimax theorem and swap the inner maximum and minimum operator which leads to the reformulation 
\[
    \min_{\substack{x\in \mathcal X \\ y^1, \ldots , y^k\in \mathcal Y(x) \\ \lambda \in \mathbb R_+^k \\ \sum_{i=1}^{k} \lambda_i = 1}} \max_{\xi\in \mathcal U}\ \sum_{i=1}^{k} \lambda_if(x,y^i,\xi).
\]
Let $(\bar x,\bar y^{1},\ldots ,\bar y^{k}, \bar \lambda)$ be an optimal solution of the latter problem and assume w.l.o.g. that $\bar \lambda_1\ge \ldots \ge \bar \lambda_k$. We define a feasible solution for the $s$-adaptability problem as
\[
x(s) = \bar x, \quad y^1(s)=\bar y^{1}, \ldots , y^s(s)=\bar y^{s},
\]
and
\[
\lambda(s)_1=\bar \lambda_1, \ldots , \lambda(s)_{s-1} = \bar \lambda_{s-1}, \ \lambda(s)_s = \sum_{i=s}^{k}\bar \lambda_i.
\]
Then we have
\begin{align*}
\text{opt}(s)-\text{opt}(k) & \le \max_{\xi\in \mathcal U}\ \sum_{i=1}^{s} \lambda(s)_i f(x(s),y^i(s),\xi) - \max_{\xi\in \mathcal U} \ \sum_{i=1}^{k} \bar \lambda_i f(\bar x,\bar y^{i},\xi).
\end{align*}
Let $\xi^*(s)$ be a scenario which maximizes the first maximum of the latter expression. Then we can further bound
\begin{align*}
    &\max_{\xi\in \mathcal U}\ \sum_{i=1}^{s} \lambda(s)_i f(x(s),y^i(s),\xi) - \max_{\xi\in \mathcal U}\ \sum_{i=1}^{k} \bar \lambda_i f(\bar x,\bar y^{i},\xi) \\
    &\le \sum_{i=1}^{s} \lambda(s)_i f(x(s),y^i(s),\xi^*(s)) - \sum_{i=1}^{k} \bar \lambda_i f(\bar x,\bar y^{i},\xi^*(s)) \\
    & = \sum_{i=s+1}^{k} \bar \lambda_i\left( f(\bar x,\bar y^{s},\xi^*(s)) - f(\bar x,\bar y^{i},\xi^*(s)) \right) \\
    & \le L \sum_{i=s+1}^{k}\bar \lambda_i\| \bar y^{s} - \bar y^{i} \| \\
    & \le L \text{diam}(\mathcal Y) \sum_{i=s+1}^{k}\bar \lambda_i,
\end{align*}
where the first inequality follows since $\xi^*(s)$ is optimal for the first maximum and feasible for the second maximum, the first equality follows from the definition of $x(s)$, $y^i(s)$ and $\lambda(s)$, the second inequality follows from the Lipschitz continuity of $f$, and the last inequality follows from the definition of the diameter. From the sorting $\bar \lambda_1\ge \ldots \ge \bar \lambda_k$ and since $\sum_{i=1}^{k}\bar \lambda_i = 1$ it follows $\bar \lambda_i\le \frac{1}{i}$. Hence we can further bound 
\begin{align*}
    L \text{diam}(\mathcal Y) \sum_{i=s+1}^{k}\bar \lambda_i \le L \text{diam}(\mathcal Y)\sum_{i=s+1}^{k}\frac{1}{i}
    \rev{\le L \text{diam}(\mathcal Y) \int_{s}^k \frac{1}{t}dt = L \text{diam}(\mathcal Y)\ln\left(\frac{k}{s}\right),}
\end{align*}
where the last inequality and equality follow from \rev{classical theory of integrals}. This proves the result.
\end{proof}
The bound in Theorem \ref{thm:approximation_bound} can depend on the dimension $n_y$, since $\text{diam}(\mathcal Y)$ can depend on $n_y$. However, it goes to zero if $s\to k$. A similar bound was derived in \cite{kurtz2024approximation} for the linear case and it was shown that the bound \rev{can be used to derive bounds on the number of policies $k$ needed to achieve certain approximation guarantees}. Similarly, for the non-linear case studied in this work, we can apply Theorem \ref{thm:approximation_bound} with $k=n_\xi+1$ \rev{to bound the difference between $\text{opt}(k)$ and the optimal value of \eqref{eq:2StageRO} for any integer $1\le k\le n_\xi+1$.} Furthermore, we can use Theorem \ref{thm:approximation_bound} to provide bounds on the number of policies $k$ which lead to a certain additive approximation guarantee $\alpha$; see Figure \ref{fig:approximation_guarantee}.

\begin{figure}
    \centering
    \includegraphics[width=0.4\linewidth]{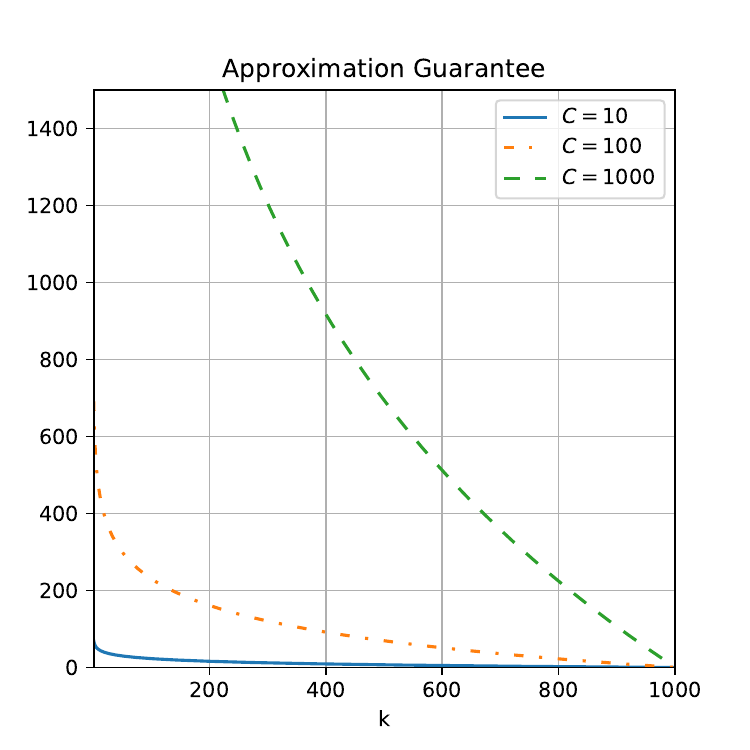}
    \caption{Plot of the additive approximation guarantee \rev{$\alpha=C\ln\left( \frac{n_\xi +1}{k}\right)$} which \eqref{eq:k-adaptability} provides for \eqref{eq:2StageRO} depending on $k$ for different constants $C$ and $n_\xi=1000$.}
    \label{fig:approximation_guarantee}
\end{figure}

\rev{\begin{corollary}
Under the assumptions of Theorem \ref{thm:approximation_bound} assume that $k\ge n_\xi+1-\frac{(n_\xi+1)(e^{\frac{\alpha}{L\text{diam}(\mathcal Y)}}-1)}{ e^{\frac{\alpha}{L\text{diam}(\mathcal Y)}}}$ and $\alpha\ge 0$. Then it holds
    \[
    \text{opt}(k) \le \text{opt}(\text{2RO}) + \alpha ,
    \]
    where $\text{opt}(\text{2RO})$ is the optimal value of \eqref{eq:2StageRO}.
\end{corollary}
\begin{proof}
From Theorem \ref{thm:n+1_policies} we know that $\text{opt}(\text{2RO})=\text{opt}(n_\xi+1)$. Define $l:=\lfloor\frac{(n_\xi+1)(e^{\frac{\alpha}{L\text{diam}(\mathcal Y)}}-1)}{ e^{\frac{\alpha}{L\text{diam}(\mathcal Y)}}}\rfloor$. Then, it follows from Theorem \ref{thm:approximation_bound}
\begin{align*}
    \text{opt}(k) - \text{opt}(\text{2RO}) & \le \text{opt}(n_\xi+1-l) - \text{opt}(n_\xi+1) \\
    &\le L \text{diam}(\mathcal Y) \ln\left( \frac{n_\xi + 1}{n_\xi + 1 - l}\right) \\ 
    & \le L \text{diam}(\mathcal Y) \ln\left( \frac{n_\xi + 1}{\frac{(n_\xi+1)(e^{\frac{\alpha}{L\text{diam}(\mathcal Y)}}) - (n_\xi+1)(e^{\frac{\alpha}{L\text{diam}(\mathcal Y)}}-1)}{e^{\frac{\alpha}{L\text{diam}(\mathcal Y)}}}}\right) \\
    & = L \text{diam}(\mathcal Y) \ln\left( e^{\frac{\alpha}{L\text{diam}(\mathcal Y)}}\right) = \alpha
\end{align*}
where the first inequality follows from $k\ge n_\xi + 1 - l$ and $\text{opt}(\text{2RO})=\text{opt}(n_\xi+1)$, the second inequality follows from Theorem \ref{thm:approximation_bound}, the third inequality follows from the definition of $l$, and the last two equalities follow from simplifying the terms.
\end{proof}}

\section{Constraint Uncertainty} \label{sec:constraint_uncertainty}
In this section we study the connection between Problems $\eqref{eq:2StageRO}$ and \eqref{eq:k-adaptability} when the uncertainty appears in the constraints and if $\mathcal U$ is convex. More precisely, we consider functions
\[
f(x,y,\xi):=\begin{cases}
    g(x,y,\xi) & \text{ if } A(\xi)x + B(\xi)y \ge h(\xi) \\
    \infty & \text{ otherwise,}
\end{cases}
\]
where $g: \mathcal X\times \mathcal Y \times \mathcal U \to \mathbb R$ is a given continuous objective function which is concave in $\xi$ and $A(\xi)\in \mathbb R^{m\times n_x}$, $B(\xi)\in \mathbb R^{m\times n_y}$ and $h(\xi)\in \mathbb R^{m}$ are the constraint parameters which are given as affine functions of the uncertain parameters. The two-stage robust problem is then given as
\begin{equation}\label{eq:2StageRO_constr}\tag{2RO-C}
    \inf_{x\in \mathcal X} \sup_{\xi\in \mathcal U} \inf_{y\in \mathcal Y(x)} \ f(x,y,\xi)
\end{equation}
and the $k$-adaptability problem is given as
\begin{equation}\label{eq:k-adaptability_constr}\tag{k-ARO-C}
    \inf_{\substack{x\in \mathcal X \\ y^1, \ldots , y^k\in \mathcal Y(x)}} \sup_{\xi\in \mathcal U} \inf_{i=1,\ldots ,k} \ f(x,y^i,\xi).
\end{equation}
Note that in contrast to the objective uncertainty case we have to use the infimum and supremum operators since for discontinuous functions $f$ we cannot guarantee that the maximum or minimum is always attained; see \cite{hanasusanto2015k} for an example. Since $\mathcal Y$ is finite at least the inner infimum could be replaced by the minimum operator, but for comprehensibility we will use the infimum operator instead.

Note that \rev{in \eqref{eq:2StageRO_constr}} for any $x\in\mathcal X$ for which a $\xi\in\mathcal U$ exists such that there exists no $y\in\mathcal Y(x)$ which is feasible for $A(\xi)x + B(\xi)y \ge h(\xi)$, the objective value is $\infty$. We call such a solution \textit{infeasible}. \rev{Similarly, for \eqref{eq:k-adaptability_constr} we call a solution $x\in\mathcal X$ infeasible if no set of $k$ second-stage solutions $y^1,\ldots ,y^k\in \mathcal Y(x)$ exists such that for every $\xi\in\mathcal U$ at least one $y^i$ is feasible for $A(\xi)x + B(\xi)y \ge h(\xi)$.} Furthermore, we assume that \eqref{eq:k-adaptability_constr} with $k=1$ always has at least one feasible solution $x$, i.e., the problems \eqref{eq:2StageRO_constr} and \eqref{eq:k-adaptability_constr} are feasible for any $k$. Since $g$ is continuous and all sets $\mathcal X, \mathcal U, \mathcal Y$ are compact, it follows that the optimal value of all the latter problems is finite.  

Unfortunately, the bounds on $k$ derived in the previous section are not valid in the constraint uncertainty case. In \cite{hanasusanto2015k} the authors provide an example where in \eqref{eq:k-adaptability_constr} all $k=|\mathcal Y|$ second-stage policies are needed to obtain an optimal solution to \eqref{eq:2StageRO_constr}. Hence, there is no hope to obtain a better bound in the general setting. However, we will derive better bounds in this section for certain problem structures.

The main idea for the results is presented in the following. Consider any fixed first-stage solution $x\in\mathcal X$ which is feasible. Following the reformulation of the proof of Theorem \ref{thm:n+1_policies} we can reformulate the inner sup-inf problem of \eqref{eq:k-adaptability_constr} for $k=|\mathcal Y(x)|$ as
\begin{equation}\label{eq:reformulation_constraint_uncertainty}
\begin{aligned}
    \sup_{z,\xi} & \ z \\
    s.t. \quad & f(x,y,\xi) - z \ge 0 \quad \forall y\in \mathcal Y(x)\\
    & \xi \in \mathcal U, \ z \in \mathbb R.
\end{aligned}
\end{equation}
Unfortunately, we cannot apply the same argumentation as in the proof of Theorem \ref{thm:n+1_policies} since now the function $f$ is not concave in $\xi$. In fact, the latter problem is a problem with up to $|\mathcal Y(x)|$ non-convex constraints and we cannot use Theorem \ref{thm:califiore} to bound the number of support constraints. However, assume we know a convex region $\mathcal D \subset \mathcal U$ for which the following holds: for every $y\in \mathcal Y(x)$, the solution $y$ is feasible for the constraint system $B(\xi)y \ge h(\xi)-A(\xi)x$ either for all $\xi\in\mathcal D$ or for no $\xi\in\mathcal D$. We call such a region \textit{recourse-stable} and we denote by $\mathcal Y_{\mathcal D}(x)$ the set of second-stage solutions in $\mathcal Y(x)$ which are feasible for all $\xi\in \mathcal D$. Note that $\mathcal D$ can be an open set. \rev{More formally, we provide the following definitions. 
\begin{definition}[Recourse Stability]
For a given $x\in\mathcal X$ a convex set $\mathcal D\subset \mathcal U$ is called recourse-stable, if
\[
\left\{ y\in \mathcal Y(x): \exists \xi\in \mathcal D \ s.t. \ A(\xi) x + B(\xi) y \ge h(\xi)\right\} = \left\{ y\in \mathcal Y(x): \forall \xi\in \mathcal D, A(\xi) x + B(\xi) y \ge h(\xi)\right\}.
\]
\end{definition}
\begin{definition}
The set of feasible second-stage solutions for a recourse-stable region $\mathcal D$ is defined as
\[
\mathcal Y_{\mathcal D}(x):=\left\{ y\in \mathcal Y(x): \forall \xi\in \mathcal D, A(\xi) x + B(\xi) y \ge h(\xi)\right\} 
\]
\end{definition}
}

If we consider Problem \eqref{eq:reformulation_constraint_uncertainty} only on a convex recourse-stable region $\mathcal D$ (instead of $\mathcal U$) then we can remove all constraints for which the corresponding second-stage solution $y$ is infeasible on $\mathcal D$ since the left-hand-side constraint value is infinity. For all others, we can replace the function $f$ by the function $g$, leading to
\begin{equation}\label{eq:reformulation_constraint_uncertainty2}
\begin{aligned}
    \sup_{z,\xi} & \ z \\
    s.t. \quad & g(x,y,\xi) - z \ge 0 \quad \forall y\in \mathcal Y_{\mathcal D}(x) \\
    & \xi \in \text{cl}\left(\mathcal D\right), \ z \in \mathbb R, 
\end{aligned}
\end{equation}
where we additionally replaced $\mathcal D$ by its closure. This can be done since $h,A,B$ are affine functions in $\xi$ and hence the set of $\xi$ which fulfill the constraints $B(\xi)y \ge h(\xi)-A(\xi)x$ for a given $y\in \mathcal Y_{\mathcal D}(x)$ is closed and contains the set $\mathcal D$. It follows that all solutions in $\mathcal Y_{\mathcal D}(x)$ are also feasible for all $\xi\in \text{cl}(\mathcal D)$ and using function $g$ instead of $f$ is valid.

Since $g$ is concave in $\xi$, Problem \eqref{eq:reformulation_constraint_uncertainty2} is a convex problem and since $g$ is continuous in $\xi$ every constraint describes a closed convex set and the supremum can be replaced by the maximum.  We can apply Theorem \ref{thm:califiore} to show that at most $n_\xi+1$ support constraints are necessary, i.e., we can remove all but $n_\xi+1$ of the second-stage policies without changing the optimal solution. Assume now we have $R$ convex recourse-stable regions $\mathcal D_1,\ldots ,\mathcal D_R\subseteq \mathcal U$ such that $\cl{\mathcal D_1}\cup\ldots\cup\cl{\mathcal D_R} = \mathcal U$. We can now apply the latter idea to every of the recourse-stable regions, which indicates that we need at most $R(n_\xi + 1)$ second-stage policies in total. Note that the recourse-stability of a region depends on the solution $x\in \mathcal X$. However, if such a cover of at most $R$ convex recourse-stable regions exists for every $x\in \mathcal X$ the previous derivation motivates the following Theorem.

% , then Problem \eqref{eq:k-adaptability} is equivalent to
% \begin{equation}\label{eq:reformulation_constraint_uncertainty2}
% \begin{aligned}
%     \max_{z,\xi^1,\ldots ,\xi^R} & \ z \\
%     s.t. \quad & z-z^i\ge 0 \quad i=1,\ldots ,R \\
%     & g(x,y,\xi^i) - z^i \ge 0 \quad \forall y\in \mathcal Y_{\mathcal D_i}(x), \ i=1,\ldots ,R\\
%     & \xi^i \in \text{cl}(\mathcal D_i) \quad i=1,\ldots ,R \\
%     & z\in\mathbb R, \ z^i \in \mathbb R\quad i=1,\ldots ,R.
% \end{aligned}
% \end{equation}
% The latter problem is a convex problem in dimension $R(n_\xi + 1)+1$ and we can apply Theorem 2 in \cite{calafiore2005uncertain} again to conclude that we can remove all but $R(n_\xi + 1)+$ of the constraints without changing the optimal solution and optimal value. 
\begin{theorem}\label{thm:num_policies_generalcover}
\rev{Let $\mathcal U$ be convex and} $g:\mathcal X\times \mathcal Y\times \mathcal U\to \mathbb R$ a continuous function such that $g(x, y, \xi)$ is concave in $\xi$ for every $x\in \mathcal X, y\in\mathcal Y$. Furthermore, assume that for every $x\in\mathcal X$ there exist $R$ convex recourse-stable regions $\mathcal D_1,\ldots ,\mathcal D_R\subseteq \mathcal U$ such that $\cl{\mathcal D_1}\cup\ldots\cup\cl{\mathcal D_R} = \mathcal U$. Then, if \[k\ge \min\left\{ R(n_\xi +1), |\mathcal Y|\right\},\] a solution $x\in\mathcal X$ is optimal for \eqref{eq:k-adaptability_constr} if and only if it is optimal for \eqref{eq:2StageRO_constr}.
\end{theorem}
\begin{proof}
Consider any fixed $x\in\mathcal X$ and $R$ convex recourse-stable regions $\mathcal D_1,\ldots ,\mathcal D_R\subseteq \mathcal U$ such that $\cl{\mathcal D_1}\cup\ldots\cup \cl{\mathcal D_R} = \mathcal U$. Since $g$ is concave in $\xi$, for every $i\in [R]$ Problem \eqref{eq:reformulation_constraint_uncertainty2} with $\mathcal D=\mathcal D_i$ is convex and since $g$ is continuous in $\xi$ every constraint corresponds to a convex closed set. Hence, we can apply Theorem \ref{thm:califiore} which shows that we can remove all constraints except $n_\xi + 1$ support constraints without changing the optimal value of the problem. For every $i\in [R]$ let $y^{i1},\dots ,y^{i(n_\xi+1)}\in \mathcal Y_{\mathcal D_i}(x)$ be the solutions related to the support constraints \rev{(assume we use duplications of the solutions in case $\mathcal Y_{\mathcal D_i}(x)$ contains less than $n_\xi+1$ solutions)}. We define now the problem 
\begin{equation}\label{eq:reformulation_constraint_uncertainty3}
\begin{aligned}
    \sup_{z,\xi} & \ z \\
    s.t. \quad & f(x,y^{ij},\xi) - z \ge 0 \quad \forall i\in [R], j\in [n_\xi+1]\\
    & \xi \in \mathcal U, \ z \in \mathbb R.
\end{aligned}
\end{equation}
which uses at most $R(n_\xi+1)$ second-stage solutions. To prove the theorem we show that the optimal value of \eqref{eq:reformulation_constraint_uncertainty3} is equal to the optimal value of \eqref{eq:reformulation_constraint_uncertainty}.

First, consider the case where $x$ is an infeasible solution for \eqref{eq:2StageRO_constr}, i.e., there exists a $\xi\in \mathcal U$ such that no $y\in \mathcal Y(x)$ is feasible for the constraint system $A(\xi)x + B(\xi)y \ge h(\xi)$. \rev{It follows that at least for one recourse-stable region $\mathcal D_i$ it must hold $\mathcal Y_{\mathcal D_i}=\emptyset$ and therefore in \eqref{eq:reformulation_constraint_uncertainty3} we cannot pick any second-stage solutions for this set and the problem must be infeasible.} Hence, the optimal value of \eqref{eq:reformulation_constraint_uncertainty} and \eqref{eq:reformulation_constraint_uncertainty3} are both $\infty$ in this case.

Now, consider the case where $x$ is a feasible solution, i.e., for every $\xi\in\mathcal U$ there exists a feasible second-stage solution. In the following we denote the optimal value of Problem \eqref{eq:reformulation_constraint_uncertainty} and \eqref{eq:reformulation_constraint_uncertainty3} as $\text{opt}\eqref{eq:reformulation_constraint_uncertainty}$ and $\text{opt}\eqref{eq:reformulation_constraint_uncertainty3}$. Since $y^{ij}\in\mathcal Y(x)$ for all $i\in[R]$ and $j\in [n_\xi+1]$, it follows that $\text{opt}\eqref{eq:reformulation_constraint_uncertainty}\le \text{opt}\eqref{eq:reformulation_constraint_uncertainty3}$. 

To show the reverse inequality let $(\xi^*,z^*)$ be an optimal solution of \eqref{eq:reformulation_constraint_uncertainty3}. Then there exists an $i^*\in[R]$ such that $\xi^*\in \cl{\mathcal D_{i^*}}$. Hence, we obtain

\begin{equation*}
    \text{opt}\eqref{eq:reformulation_constraint_uncertainty3} \ = \ \begin{aligned}
        \sup_{z,\xi} \ & z \\
        s.t. \quad & f(x,y^{ij},\xi)-z\ge 0 \quad i\in [R], j\in [n_\xi +1]\\
        & \xi \in \cl{\mathcal D_{i^*}}, \ z\in\R,
    \end{aligned}
\end{equation*}
which is smaller or equal to
\begin{equation*}
    \begin{aligned}
        \sup_{z,\xi} \ & z \\
        s.t. \quad & f(x,y^{i^*j},\xi)-z\ge 0 \quad j\in [n_\xi +1]\\
        & \xi \in \cl{\mathcal D_{i^*}}, \ z\in\R.
    \end{aligned}
\end{equation*}
The optimal value of the last problem is equal to the optimal value of
\begin{equation}\label{eq:reformulation_constraint_uncertainty_4}
    \begin{aligned}
        \sup_{z,\xi} \ & z \\
        s.t. \quad & g(x,y^{i^*j},\xi)-z\ge 0 \quad j\in [n_\xi +1]\\
        & \xi \in \cl{\mathcal D_{i^*}}, \ z\in\R,
    \end{aligned}
\end{equation}
since for every $\xi\in \mathcal D_{i^*}$ it holds $f(x,y^{i^*j},\xi)=g(x,y^{i^*j},\xi)$ since $y^{i^*j}\in \mathcal Y_{\mathcal D_{i^*}}(x)$. The same can be shown for $\xi\in\cl{\mathcal D_{i^*}}\setminus \mathcal D_{i^*}$ since $h,A,B$ are affine functions in $\xi$ and hence the set of $\xi$ which fulfill the constraints $B(\xi)y \ge h(\xi)-A(\xi)x$ for a given $y\in \mathcal Y_{\mathcal D_{i^*}}(x)$ is closed and contains the set $\mathcal D_{i^*}$. Hence, it must contain $\cl{\mathcal D_{i^*}}$ and we can conclude that all solutions in $\mathcal Y_{\mathcal D_{i^*}}(x)$ are also feasible for all $\xi\in \text{cl}(\mathcal D_{i^*})$ and we can use function $g$ instead of $f$ on the whole closure.

By the definition of the support constraints the optimal value of \eqref{eq:reformulation_constraint_uncertainty_4} is equal to
\begin{equation}\label{eq:reformulation_constraint_uncertainty_5}
    \begin{aligned}
        \sup_{z,\xi} \ & z \\
        s.t. \quad & g(x,y,\xi)-z\ge 0 \quad y\in\mathcal Y_{D_{i^*}}(x)\\
        & \xi \in \cl{\mathcal D_{i^*}}, \ z\in\R.
    \end{aligned}
\end{equation}
If we can show that any optimal solution of \eqref{eq:reformulation_constraint_uncertainty_5} is feasible for \eqref{eq:reformulation_constraint_uncertainty} then we proved $\text{opt}\eqref{eq:reformulation_constraint_uncertainty3}\le \text{opt}\eqref{eq:reformulation_constraint_uncertainty}$. Consider any optimal solution $(\bar \xi,\bar z)$ for \eqref{eq:reformulation_constraint_uncertainty_5} with $\bar \xi\in \mathcal D_{i^*}$. Then it holds that $f(x,y,\bar \xi)=g(x,y,\bar \xi)$ for every $y\in\mathcal Y_{\mathcal D_{i^*}}(x)$ and $f(x,y,\bar \xi)=\infty$ otherwise. From feasibility for \eqref{eq:reformulation_constraint_uncertainty_5} it follows that the corresponding solution must be feasible for \eqref{eq:reformulation_constraint_uncertainty}. Now consider the remaining case where $\bar \xi\in \cl{\mathcal D_{i^*}}\setminus \mathcal D_{i^*}$. Then there exists an infinite sequence $\{\bar \xi_t\}_{t\in \N}$ with $\bar \xi_t\in\mathcal D_{i^*}$ and $\lim_{t\to\infty} \bar \xi_t = \bar \xi$. Set $\bar z_t = \min_{y\in \mathcal Y_{\mathcal D_{i^*}}(x)} g(x,y,\bar \xi_t)$. Then $(\bar \xi_t,\bar z_t)$ is feasible for \eqref{eq:reformulation_constraint_uncertainty_5} for every $t$ and $\lim_{t\to\infty} (\bar \xi_t,\bar z_t) = (\bar \xi, \bar z)$. Since $\bar \xi_t\in\mathcal D_{i^*}$ for every $t$, by the discussion above every $(\bar \xi_t,\bar z_t)$ is feasible for \eqref{eq:reformulation_constraint_uncertainty} and hence, the optimal value of \eqref{eq:reformulation_constraint_uncertainty} must be at least $\bar z$ which is the optimal value of \eqref{eq:reformulation_constraint_uncertainty_5}. This shows $\text{opt}\eqref{eq:reformulation_constraint_uncertainty}\ge \text{opt}\eqref{eq:reformulation_constraint_uncertainty3}$ and we proved $\text{opt}\eqref{eq:reformulation_constraint_uncertainty} = \text{opt}\eqref{eq:reformulation_constraint_uncertainty3}$.

In summary we showed that for every $x\in\mathcal X$ there exist at most $R(n_\xi+1)$ second-stage solutions such that Problems \eqref{eq:reformulation_constraint_uncertainty} and \eqref{eq:reformulation_constraint_uncertainty3} have the same optimal value which proves the result.
\end{proof}

As for the bound derived in Theorem \ref{thm:n+1_policies}, \rev{the bound on $k$ in Theorem \ref{thm:num_policies_generalcover} only depends on the dimension of the uncertain parameters}. However, the dimension $n_y$ may be hidden in the number $R$ as we will see in the following section. Note, again no convexity is required for $g$ regarding the variables $x$ and $y$.

\begin{remark}\label{rem:objective_independent_of_uncertainty}
If the objective function $g$ does not depend on $\xi$, i.e., $g(x,y,\xi)=\bar g(x,y)$, then for each recourse-stable region $\mathcal D_i$ Problem \eqref{eq:reformulation_constraint_uncertainty2} is equivalent to
\[\max_{\xi\in \cl{\mathcal D_i}} \min_{y\in \mathcal Y_{\mathcal D_i}(x)} \bar g(x,y),\]
\rev{and since the inner minimum problem does not depend on $\xi$ we can drop the maximum operator.} Hence, there exists a single solution in $\mathcal Y_{\mathcal D_i}$ which minimizes $\bar g(x,y)$  for every $\xi\in \cl{\mathcal D_i}$. It follows that in this case the bound on $k$ from Theorem \ref{thm:num_policies_generalcover} can be improved to
\[
k\ge \min\left\{ R, |\mathcal Y|\right\}.
\]
\end{remark}

Finally, we can show that for constraint-wise uncertainty and fixed-recourse we only need $k=1$ second-stage solutions for optimality. A similar result was shown for continuous decisions in \cite{marandi2018static}.

% \begin{lemma}
% Assume the objective function $g$ does not depend on $\xi$, i.e., $g(x,y,\xi)=\bar g(x,y)$. Furthermore, assume that for every $x\in\mathcal X$ there exist $R$ convex recourse-stable regions $\mathcal D_1,\ldots ,\mathcal D_R\subseteq \mathcal U$ such that $\cl{\mathcal D_1}\cup\ldots\cup\cl{\mathcal D_R} = \mathcal U$ and such that there exists a partition of $[R]$ into $T$ subsets, i.e., 
% \[
% [R] = \bigcup_{i=1}^{T} \{ \nu^i_1, \ldots ,\nu^i_{t_i}\}
% \]
% such that $\mathcal Y_{\mathcal D_{\nu^i_1}}(x)\subseteq \mathcal Y_{\mathcal D_{\nu^i_2}}(x)\subseteq \cdots \subseteq \mathcal Y_{\mathcal D_{\nu^i_{t_i}}}(x)$ for every $i\in [T]$. Then, if \[k\ge \min\left\{ T, |\mathcal Y|\right\},\] a solution $x\in\mathcal X$ is optimal for \eqref{eq:k-adaptability_constr} if and only if it is optimal for \eqref{eq:2StageRO_constr}. 
% \end{lemma}
% \begin{proof}
% The idea of the proof is the same as for Theorem \ref{thm:num_policies_generalcover}. However, due to the structure of the recourse-stable regions we can exclude a set of regions which have to be considered: Since $\mathcal Y_{\mathcal D_{\nu^i_1}}(x)\subseteq \mathcal Y_{\mathcal D_{\nu^i_2}}(x)\subseteq \cdots \subseteq \mathcal Y_{\mathcal D_{\nu^i_{t_i}}}(x)$ for every $i\in [T]$ and since $g$ does not depend on $\xi$, we know that 
% \[
% \sup_{\xi \in \mathcal D_{\nu^i_j}}\min_{y\in\mathcal Y(x)} g(x,y) \le \sup_{\xi \in \mathcal D_{\nu^i_{t_i}}}\min_{y\in\mathcal Y(x)} g(x,y)
% \]
% \end{proof}

\begin{corollary}
Assume $g$ does not depend on $\xi$, $B(\xi)=B$ for all $\xi\in \mathcal U$ and the uncertainty appears constraint-wise, i.e., we consider the problem
\[
\inf_{x\in\mathcal X} \sup_{\xi^1,\ldots ,\xi^m\in\mathcal U} \inf_{y\in\mathcal Y(x)} f(x,y,\xi^1,\ldots ,\xi^m)
\]
where the constraints are given as
\[
a_i(\xi^i)^\top x + b_i^\top y \ge h_i(\xi^i) \quad i\in [m].
\]
Then, for $k=1$ a solution for $\eqref{eq:k-adaptability_constr}$ is optimal if and only if it is optimal for \eqref{eq:2StageRO_constr}.
\end{corollary}
\begin{proof}
    Since for every $x\in\mathcal X$ the function $h_i(\xi^i) - a_i(\xi^i)^\top x$ is continuous in $\xi^i$ and $\mathcal U$ compact, for every $i\in [m]$ there exists $\bar \xi^i\in\mathcal U$ which maximizes the latter function over $\mathcal U$ . This scenario leads to the smallest number of feasible second-stage solutions and is hence a maximizing scenario. The problem reduces to
    \begin{align*}
    \inf_{x,y} & \ g(x,y) \\
    s.t. \quad & a_i(\bar \xi^i)^\top x + b_i^\top y \ge h_i(\bar \xi^i) \quad i\in [m] \\
    & x\in \mathcal X, y\in \mathcal Y(x)
    \end{align*}
    which shows that only $k=1$ solutions are needed.
\end{proof}

The main task in the following is to bound the number of recourse-stable regions for certain problem structures to obtain good values for $R$.

\subsection{Bounds on the Number of Recourse-Stable Regions}
In this section we derive bounds on the number of recourse-stable regions which are needed to cover the uncertainty set $\mathcal U$. By Theorem \ref{thm:num_policies_generalcover} we obtain then a bound on the number of policies $k$ which are needed to get an optimal solution for \eqref{eq:2StageRO_constr}.

We assumed that all constraint parameters are given as affine functions of the uncertain parameters $\xi$, i.e., we have
\[
h(\xi) = h + H\xi, \quad A(\xi) = A + \sum_{i=1}^{n_\xi} A^i\xi_i, \quad B(\xi) = B + \sum_{i=1}^{n_\xi} B^i\xi_i
\]
where $h\in \Z^m$, $H\in \Z^{m\times n_\xi}$, $A, A^i\in \Z^{m\times n_x}$, $B, B^i\in \Z^{m\times n_y}$ for all $i\in[n_\xi]$ are given parameters. \rev{Note that the constraint parameters are assumed to be integer which can always be achieved in practice by rescaling each constraint with the least common multiple of the coefficients denominators. The integrality of the coefficients will ease the analysis below.} We can reformulate the constraint system $A(\xi)x + B(\xi)y \ge h(\xi)$ as 
\begin{equation}\label{eq:polyhedral_description_feasible_constraints}
\sum_{i=1}^{n_\xi} (A^ix + B^iy-H_i)\xi_i \ge h-Ax-By 
\end{equation}
where $H_i$ is the $i$-th column of $H$. \rev{For given $x\in \mathcal X$ and $y\in \mathcal Y(x)$ the latter inequality system, defined by $m$ half-spaces intersected with $U$, describes the set of uncertainty realizations for which the solution $(x, y)$ is feasible.} 

The main idea to derive the results of this section is the following: if we can bound the number of hyperplanes in the $\xi$-space which can appear (over all different second-stage solutions $y$) in \eqref{eq:polyhedral_description_feasible_constraints} then we can bound the number of regions which are enclosed by hyperplanes and which are not intersected by any other hyperplane. We will show that the interior of each of these regions is a convex recourse-stable region and taking the union of the closures of all these regions defines a cover for $\mathcal U$. Then we can apply Theorem \ref{thm:num_policies_generalcover} to get a bound on $k$.

Fix any $x\in\mathcal X$ and define the set of all possible hyperplanes appearing in the constraints in \eqref{eq:polyhedral_description_feasible_constraints} over all $y\in\mathcal Y(x)$ which intersect with $\mathcal U$ as
\[
\mathcal H(x) := \left\{ P_i(y)=\{\xi: a_i(x,y)^\top \xi = h_i(x,y)\} : P_i(y)\cap \mathcal U\neq \emptyset, \ y\in \mathcal Y(x), i=1,\ldots ,m\right\}
\]
where $a_i(x,y)$ is the $i$-th row of the matrix \[A(x,y) := \left(A^1x + B^1y-H_1,\ldots, A^{n_\xi}x + B^{n_\xi}y-H_{n_\xi} \right)\] and $h_i(x,y)$ is the $i$-th entry of the vector $h(x,y):=h-Ax-By$. \rev{Let $\mathcal X_{\text{feas}}$ be the set of solutions in $\mathcal X$ which are feasible for \eqref{eq:2StageRO_constr}.} Define the maximum number of hyperplanes over all $x$ which are \rev{feasible for \eqref{eq:2StageRO_constr}} as \rev{
\[
\eta:=\max\{ 1, \max_{x\in \mathcal X_{\text{feas}}} |\mathcal H(x)|\},
\]
where the definition ensures that in case no hyperplane intersects with the uncertainty set, i.e., $|\mathcal H(x)|=0$, we set $\eta=1$.}

Note that $\eta$ can be significantly smaller than $|\mathcal Y|$, e.g. if $B,B^i$ are matrices with integer values. Then $\eta$ can be bounded by terms in the size of the numbers in $B,B^i$ which we will discuss later in more detail. Consider Example \ref{ex:plot_recourse-stable_regions} to motivate the results of this section.

\begin{example} \label{ex:plot_recourse-stable_regions}
Consider the problem without first-stage solutions where the second-stage feasible region is given as
\[
\mathcal Y = \left\{ y\in \{ 0,1\}^2: -y_1+\xi_2y_2 \le \xi_1, \ y_1+3y_2 \ge \xi_2\right\}
\]
and uncertainty set $\mathcal U=[\frac{3}{2},\frac{7}{2}] \times [\frac{1}{2},2]$. We can go through all possible second-stage solutions in $\{0,1\}^2$ and draw the corresponding hyperplanes of the two constraints in $\mathcal Y$; see Figure \ref{fig:recourse_stable_regions}. The interior of the resulting regions are all recourse-stable. Only the ones which intersect with $\mathcal U$ are relevant. The feasible second-stage solutions for each region are given as
\[\mathcal Y_{\mathcal D_1}=\left\{  \begin{pmatrix} 1 \\ 1 \end{pmatrix}\right\}, \quad \mathcal Y_{\mathcal D_2}=\left\{ \begin{pmatrix} 0 \\ 1 \end{pmatrix}, \begin{pmatrix} 1 \\ 1 \end{pmatrix}\right\}, \quad \mathcal Y_{\mathcal D_3}=\left\{  \begin{pmatrix} 1 \\ 0 \end{pmatrix}, \begin{pmatrix} 0 \\ 1 \end{pmatrix}, \begin{pmatrix} 1 \\ 1 \end{pmatrix}\right\}.\]
\end{example}
\begin{figure}
    \centering
    \includegraphics[width=0.8\linewidth]{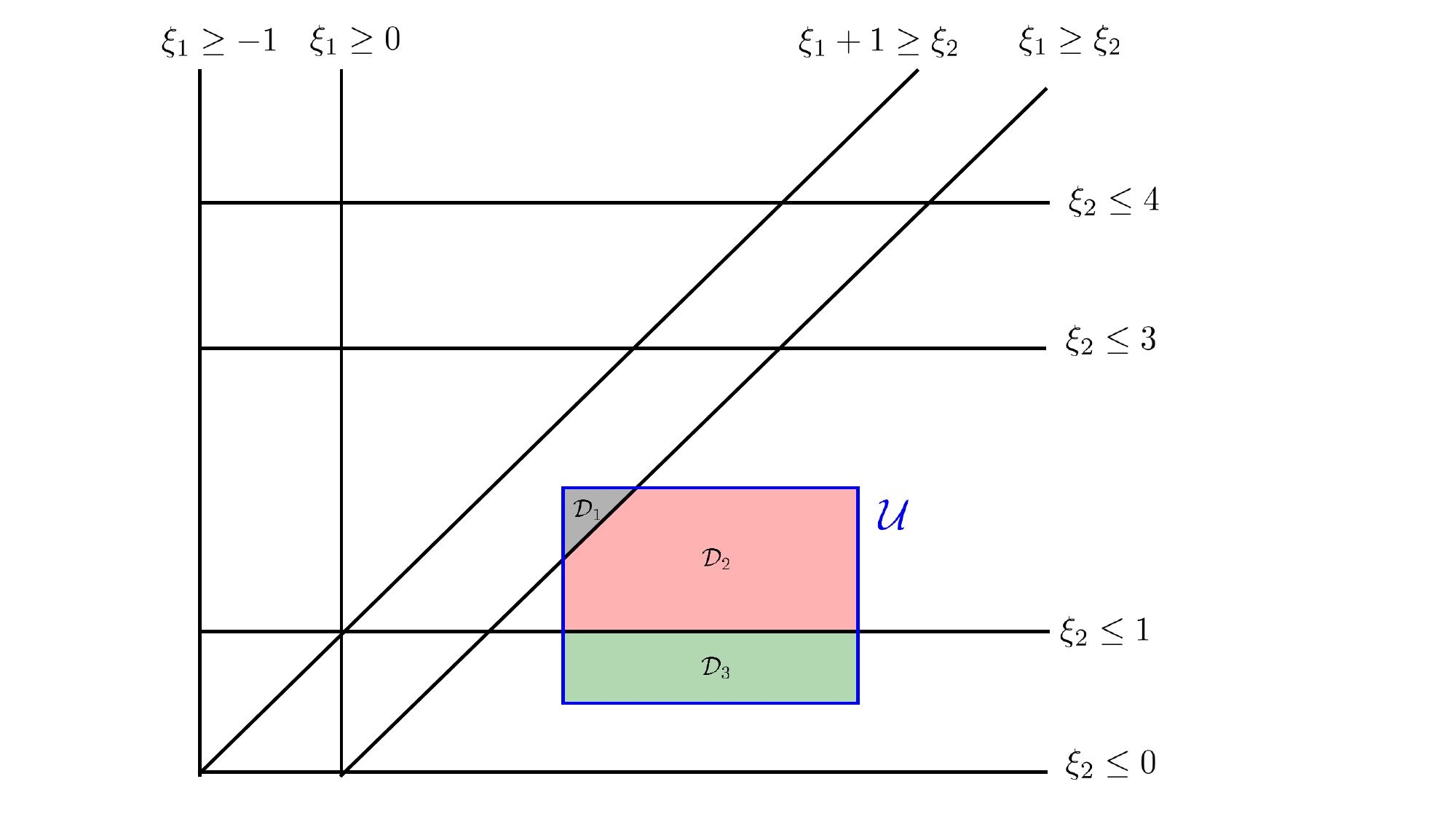}
    \caption{All hyperplanes in $\mathcal H(x)$ and corresponding recourse stable regions for Example \ref{ex:plot_recourse-stable_regions}.}
    \label{fig:recourse_stable_regions}
\end{figure}

We can prove now the following lemma.

\begin{lemma}\label{lem:bound_regions_constraint_uncertainty}
For every $x\in\mathcal X$ \rev{feasible for \eqref{eq:2StageRO_constr}} there exist at most $R\in \mathcal O\left(\eta ^{n_\xi}\right)$ convex recourse-stable regions such that the union of its closures covers $\mathcal U$.
\end{lemma}
\begin{proof}
Let $x\in\mathcal X$ be an arbitrary feasible solution for \eqref{eq:2StageRO_constr}. Consider $\mathcal H(x)$ which contains at most $\eta$ hyperplanes. This set of hyperplanes induces a set of full-dimensional regions, where each region is enclosed by a subset of these hyperplanes and no other hyperplane in $\mathcal H(x)$ is intersecting the region. It was shown in Lemma 4 in \cite{serra2018bounding} that this number of regions can be bounded by
\[
\sum_{i=0}^{\text{rank}(V)} \binom{\eta}{i} \in \mathcal O \left(\eta^{\text{rank}( V)}\right),
\]
where $V$ is the matrix which results from the concatenation of all row vectors $a_i(x,y)$ appearing in $\mathcal H(x)$. We can bound the rank of $V$ by $n_\xi$ which yields the number of regions stated in the lemma.

Next, we have to show that the interior of each of the induced regions is in fact a convex recourse-stable region. Convexity follows since each region is polyhedral. To show recourse-stability consider any full-dimensional region $\bar{\mathcal D}$ induced by the set of hyperplanes in $\mathcal H(x)$. For any $\xi$ in the interior of this region consider any $y\in \mathcal Y(x)$ which is feasible for this $\xi$. Then $\xi$ lies in the polyhedron \rev{described by the linear constraints in} \eqref{eq:polyhedral_description_feasible_constraints} for this $y$. Since none of the hyperplanes which defines \eqref{eq:polyhedral_description_feasible_constraints} intersects the interior of $\bar{\mathcal D}$, the whole interior of the region must be contained in \rev{this polyhedron} and hence $y$ is feasible for all $\xi$ in the interior of $\bar{\mathcal D}$. Hence, we proved that the interior of each region is recourse-stable.

Finally, we have to show that the union of the closures of these regions covers $\mathcal U$. This is trivially the case since we actually bounded the number of regions to cover the whole space $\R^{n_\xi}$.
\end{proof}
The bound on the size of the cover can be improved if we consider fixed recourse, i.e., the recourse matrix $B(\xi)$ does not depend on $\xi$.

\begin{lemma}\label{lem:bound_fixed_recourse}
Assume fixed recourse, i.e., $B(\xi)=:B$ for all $\xi \in \mathcal U$. For every $x\in\mathcal X$ \rev{feasible for \eqref{eq:2StageRO_constr}} there exists a cover of at most $R\in \mathcal O\left(\eta ^{\min\{m,n_\xi\}}\right)$ convex recourse-stable regions for $\mathcal U$.
\end{lemma}
\begin{proof}
In the case of fixed recourse, we have $B^i=0$ for all $i\in [n_\xi]$. Hence, $A(x,y)$ is the same matrix for every $y \in\mathcal Y(x)$. Following the proof of Lemma \ref{lem:bound_regions_constraint_uncertainty} the number of regions for the cover can be bounded by $\mathcal O \left(\eta^{\text{rank}(V)}\right)$ where $V$ is the matrix derived from concatenating the same matrix $A(x,y)$ multiple times. Hence, $V$ has rank at most $\min\{m,n_\xi\}$. Following the rest of the proof of Lemma \ref{lem:bound_regions_constraint_uncertainty} proves the result.
\end{proof}

We can summarize the latter results now in the following theorem.
\begin{theorem}\label{thm:bounds_k_constraint_uncertainty}
\rev{Let $\mathcal U$ be convex and} $g:\mathcal X\times \mathcal Y\times \mathcal U\to \mathbb R$ a continuous function such that $g(x, y, \xi)$ is concave in $\xi$ for every $x\in \mathcal X, y\in\mathcal Y$. Then, the number of second-stage policies needed in \eqref{eq:k-adaptability_constr} to ensure an optimal solution for \eqref{eq:2StageRO_constr} is
\begin{itemize}
    \item $k\in  O\left(\eta^{\min\{m,n_\xi\}}(n_\xi +1)\right)$ if fixed recourse holds 
    \item $k\in \mathcal O\left(\eta^{n_\xi}(n_\xi +1)\right)$ if random recourse holds. 
\end{itemize}
\end{theorem}
\begin{proof}
    The result directly follows from Lemma \ref{lem:bound_regions_constraint_uncertainty} and \ref{lem:bound_fixed_recourse} together with Theorem \ref{thm:num_policies_generalcover}.
\end{proof}

\rev{
\begin{remark}
\textcolor{black}{By Remark \ref{rem:objective_independent_of_uncertainty} it follows that, if $g$ does not depend on $\xi$, we can improve the bounds to $k\in  O\left(\eta^{\min\{m,n_\xi\}}\right)$ if fixed recourse holds and $k\in \mathcal O\left(\eta^{n_\xi}\right)$ if random recourse holds.} 
\end{remark}
}

\rev{
\begin{remark}\label{rem:bound_k_constraint_rank}
In the case of fixed recourse, the exponent $\min\{ m,n_\xi\}$ appearing in Theorem \ref{thm:bounds_k_constraint_uncertainty} is a (potentially) loose bound for the rank of the matrix
\[
A(x) := \left(A^1x -H_1,\ldots, A^{n_\xi}x -H_{n_\xi} \right)
\] 
used in the proof of Lemma \ref{lem:bound_fixed_recourse}. To achieve tighter bounds on $k$ we can use
\[
\omega:= \max_{x\in \mathcal X} \text{rank}\left( A(x)\right) \le \min\{ m,n_\xi\},
\]
leading to the bound $k\in \mathcal O\left( \eta^\omega (n_\xi +1) \right)$ in case of fixed recourse which reduces to $k\in \mathcal O\left( \eta^\omega\right)$ if the objective function does not depend on $\xi$.
\end{remark}
The following example shows, that the bound of Remark \ref{rem:bound_k_constraint_rank} is tight.
\begin{example}
Consider the following $k$-adaptability problem without first-stage variables $x$,
\[
\min_{y^1,\ldots ,y^k\in \mathcal Y} \max_{\xi\in \mathcal U} \min_{i=1,\ldots ,k} f(y)
\]
where $\mathcal U=[0,n_y]^{n_\xi}$, $\mathcal Y=\{ 0,1\}^{n_y}$ and 
\[
f(y,\xi):=\begin{cases}
0 & \text{if } \ \sum_{i=1}^{n_y} y_i \le \xi_1\le \sum_{i=1}^{n_y} y_i  + 1 \\
\infty & \text{otherwise.}
\end{cases}
\]
We can rewrite the constraint system as
\begin{align*}
\xi_1 & \ge \sum_{i=1}^{n_y} y_i \\
-\xi_1 & \ge -\sum_{i=1}^{n_y} y_i + 1
\end{align*}
and since $y_i$ are binary we have $\sum_{i=1}^{n_y} y_i\in \{ 0,1,\ldots ,n_y\}$. Hence, the latter constraint system produces $\eta=n_y+1$ hyperplanes which intersect with $\mathcal U$. At the same time the rank $\omega$ defined in Remark \ref{rem:bound_k_constraint_rank} is one and since the objective function does not depend on $\xi$ we obtain the bound $k\in \mathcal O\left( n_y\right)$. Note that every set 
\[
\mathcal D_j:=\left\{ \xi\in \mathcal U: \xi_1 \in [j-1, j)\right\}, \ j=1,\ldots ,n_y+1,
\]
is a recourse-stable region where the set of feasible solutions is
\[
\mathcal Y_{\mathcal D_j}=\left\{ y\in \{ 0,1\}^{n_y}: \sum_{i=1}^{n_y} y_i = j-1\right\} , \ j=1,\ldots ,n_y+1.
\]
If we pick one feasible solution for each recourse-stable region we achieve an objective value of $0$ which is optimal. If we do not choose a feasible solution for at least one of the recourse-stable regions the objective value is $\infty$. Hence we need 
$k^*=n_y+1$ many solutions for optimality which coincides with the bound from Remark \ref{rem:bound_k_constraint_rank}.
\end{example}
\begin{example}
To show that the bound is even tight if the objective function depends on $\xi$ we adapt the latter example. The idea is to use Example \ref{ex:bound_k_tight_objective} for every of the recourse-stable regions of the latter example. Consider $\mathcal U' = \left\{ \xi\in [0,n_y]: \sum_{i=2}^{n_\xi} \xi_i = 1\right\}$ and
$$\mathcal Y'=\left\{ (y,z)\in \{ 0,1\}^{n_y} \times \{ 0,1\}^{n_\xi -1}: \sum_{j=1}^{n_\xi -1} z_j=1 \right\}$$
and 
\[
f(y,z,\xi):=\begin{cases}
\sum_{j=1}^{n_\xi-1} \xi_{j+1} z_j  & \text{if } \ \sum_{i=1}^{n_y} y_i \le \xi_1\le \sum_{i=1}^{n_y} y_i  + 1 \\
\infty & \text{otherwise.}
\end{cases}
\]
Since the constraints did not change, we have the same $n_y+1$ recourse-stable regions on the $y$-variables as in the latter example. By following Example \ref{ex:bound_k_tight_objective}, to be optimal, on each of the regions we have to choose $n_\xi-1$ different second stage solutions $(y^j,e_1), (y^j,e_2), \ldots ,(y^j,e_{n_\xi-1} )$ where $y^j$ is feasible for $\mathcal D_j$. This leads to $k\in \mathcal O\left( n_\xi n_y\right)$ many second-stage solutions needed to achieve optimality which coincides with the bound of Remark \ref{rem:bound_k_constraint_rank}.
\end{example}
}

\rev{\paragraph*{Procedure to Calculate a Bound on $\eta$}
Calculating the value $\eta$ exactly is tricky since it depends on the specific problem structure. In the following we present a simple procedure to get valuable bounds on $\eta$ for the case that $\mathcal X\subset \mathbb Z^{n_x}$.
\begin{enumerate}
\item Scale the given constraint system
\[
\sum_{i=1}^{n_\xi} (A^ix + B^iy-H_i)\xi_i \ge h-Ax-By 
\]
such that all entries of $A^i$,$A$,$B$,$B^i$, $H_i$ and $h$ are integer.
\item Calculate for every index $i\in [n_\xi]$ the values
\begin{align*}
\bar v_i & := \max_{l\in [m]} \left( \max_{x\in \mathcal X} x^\top a_l^i + \max_{y\in\mathcal Y} y^\top b_l^i - H_{li}\right) \\
\underline{v_i} & := \min_{l\in [m]} \left( \min_{x\in \mathcal X} x^\top a_l^i + \min_{y\in\mathcal Y} y^\top b_l^i - H_{li}\right)
\end{align*}
where $a_l^i, b_l^i$ denote the $l$-th row of the corresponding matrices $A^i$ and $B^i$.
\item Calculate the the right-hand-side value
\begin{align*}
\bar v_0 & := \max_{l\in [m]} \left( h_l + \max_{x\in \mathcal X} -x^\top a_l + \max_{y\in\mathcal Y} -y^\top b_l \right) \\
\underline{v_0} & := \min_{l\in [m]} \left( h_l + \min_{x\in \mathcal X} -x^\top a_l + \min_{y\in\mathcal Y} -y^\top b_l \right) 
\end{align*}
where $a_l, b_l$ denote the $l$-th row of the corresponding matrices $A$ and $B$.
\item Then it follows $\eta \le \prod_{i=0}^{n_\xi}( \bar v_i - \underline{v_i} +1)$
\end{enumerate}
The idea of the algorithm is to calculate the maximum and the minimum value of the coefficients of each variable $\xi_i$ in the constraint system which can appear over all $x\in\mathcal X$ and $y\in\mathcal Y$ (Step 2). Due to the scaling in Step 1, and since $x$ and $y$ have only integer values, all these coefficient values have to be integer and hence the number of possible coefficients of variable $\xi_i$ is $\bar v_i - \underline{v_i} +1$. After the same argumentation for the right-hand-side (Step 3) we calculate all combinations of parameter values in Step 4, which leads to a valid bound for $\eta$.

Note that in case $\mathcal X$ is not restricted to integer values, a similar procedure can be used which only calculates the number of values which the terms $B^iy$ and $By$ can take in each row $l$ by calculating
\begin{equation}\label{eq:eta_procedure_betai_definition}
\beta^i_l := \max_{y\in \mathcal Y} y^\top b_l^i - \min_{y\in\mathcal Y}y^\top b_l^i + 1 \ \forall i\in[n_\xi], \quad \beta^0_l:=\max_{y\in \mathcal Y} y^\top b_l - \min_{y\in\mathcal Y}y^\top b_l + 1.
\end{equation}
Taking all combinations of coefficients leads to the number of hyperplanes $\prod_{i=0}^{n_\xi}\beta^i_l$ which can appear in row $l$. However, in this case we have to sum up the number of possible hyperplanes over all rows, since we do not have information about the full coefficient value, which would involve $x$, resulting in $\eta \le \sum_{l=1}^{m}\prod_{i=0}^{n_\xi}\beta^i_l$. While the latter method could also be applied to the case of integer $x$, it leads to worse bounds if the number of constraints in the constraint system is large.

\begin{example}\label{ex:rhs-number-of-values}
Assume $\mathcal Y\subseteq \Z_+^{n_y}\cap [0,u]$ where $u\in \Z_+^{n_y}$ are given upper bounds on the second-stage decision variables. Furthermore, assume $B, B^i\in \Z^{n_y}$ for all $i$ which can be achieved by scaling the constraint system. If $\mathcal X$ is mixed-integer we bound the values \eqref{eq:eta_procedure_betai_definition} as
\[
\beta^i_l = (2 u^\top |b_l^i| +1) \ \forall i\in [n_\xi], \quad \beta^0_l =  2 (u^\top |b_l| + 1)
\]
where $|b_l^i|$ denotes the vector containing the absolute values of each entry. This leads to $\prod_{i=0}^{n_\xi} \beta^i_l$ many hyperplanes for constraint $l$. Hence, in total we can bound 
\[
\eta^{rr} \le \sum_{l=1}^{m} \prod_{i=0}^{n_\xi} \beta^i_l
\]
in case of random recourse. In case of fixed recourse we have $b_l^i = 0$ and hence $\beta^i_l=1$ for all $i\in [n_\xi]$ and we obtain
\[
\eta^{fr} \le \sum_{l=1}^{m}  \beta^0_l.
\]
\end{example}
}

We now present applications studied in other works and apply the previous results.

\begin{example}
Consider the capital budgeting problem with constraint uncertainty as defined in Section \ref{sec:problem_definitions}. We assume that all entries of $\Phi$ and $c^0$ are integer, which can be obtained after scaling.

\rev{We can reformulate the budget constraint as 
\begin{equation}\label{eq:constraint_capital_budgeting}
  \sum_{j=1}^{\rho} \left( \sum_{i=1}^{n} c_i^0(x_i+y_i)\Phi_{ij}\right) \xi_j \le B - \sum_{i=1}^{n} c_i^0(x_i+y_i)  
\end{equation}
To apply Theorem \ref{thm:bounds_k_constraint_uncertainty} we have to calculate $\eta$. Since $\mathcal X\subseteq\{ 0,1\}^n$ we can use the $4$-step procedure above to get a bound. Step 2 results in the values
\begin{align*}
\bar v_i & = \max_{x\in\mathcal X} \sum_{i=1}^{n} c_i^0 \Phi_{ij}x_i + \max_{y\in\mathcal Y} \sum_{i=1}^{n} c_i^0 \Phi_{ij}y_i \le 2 \sum_{i=1}^{n} |c_i^0 \Phi_{ij}| \\
\underline v_i &= \min_{x\in\mathcal X} \sum_{i=1}^{n} c_i^0 \Phi_{ij}x_i + \min_{y\in\mathcal Y} \sum_{i=1}^{n} c_i^0 \Phi_{ij}y_i \ge -2 \sum_{i=1}^{n} |c_i^0 \Phi_{ij}|
\end{align*}
and similarly in Step 3 we obtain
\begin{align*}
\bar v_0 & \le B + 2\sum_{i=1}^{n}|c_i^0| \\
\underline{v_0} & \ge B - 2\sum_{i=1}^{n}|c_i^0|
\end{align*}
which after applying Step 4 results in 
\[
\eta \le \underbrace{\left(4\sum_{i=1}^{n} |c_i^0| +1 \right)}_{=\bar c_0}\prod_{i=1}^{\rho} \underbrace{\left(4\sum_{i=1}^{n} |c_i^0 \Phi_{ij}| +1 \right)}_{\le \bar c}.
\]
From Theorem \ref{thm:bounds_k_constraint_uncertainty} it follows that we need at most $k\in \mathcal O\left((\bar c_0 \bar c^\rho)^\rho(\rho+1)\right)$ second-stage policies to ensure optimality. From Remark \ref{rem:objective_independent_of_uncertainty} it follows, that if the objective parameters are not uncertain this bound improves to $k\in \mathcal O\left((\bar c_0 \bar c^\rho)^\rho\right)$. Note again that $\rho$ is usually a fixed and small number and the actual number of possible second-stage solutions can be $|\mathcal Y|=2^{n}$. Hence, for a fixed number of risk factors the number of second-stage solutions grows exponential in the dimension $n$ while the bound on $k$ remains constant.
}
\end{example}

\begin{example}
\rev{
Consider a capacitated facility location problem as defined in Section \ref{sec:problem_definitions}. To calculate $\eta$ we apply the $4$-step procedure above. We can reformulate the constraint system which involves the uncertain parameters as
\[
\sum_{l=1}^{n_\xi} (c_{l} x_i)\xi_l \ge \sum_{j\in\mathcal J} d_jy_{ij} \quad \forall i\in \mathcal I.
\] 
Since $x_i\in \{ 0,1\}$ the coefficients of the left side can only be all zero or fixed to $c$. In the first case the right-hand-side is also zero and hence we can ignore these constraints. Hence, we only have one option for the left-hand-side coefficients. Since for any $x\in \mathcal X$ at most $p$ entries of $x$ can be non-zero, the number of constraints reduces to $p$. For the right-hand-side we obtain 
\begin{align*}
\bar v_0 & = \max_{i\in \mathcal I}\max_{y\in \mathcal Y} \sum_{j\in\mathcal J} d_jy_{ij} \le  D \\
\underline{v_0} & =  \min_{i\in \mathcal I}\min_{y\in \mathcal Y} \sum_{j\in\mathcal J} d_jy_{ij} = 0
\end{align*} 
where $D = \sum_{j\in\mathcal J}d_j$. We obtain in Step 4
\[
\eta \le D+1
\]
and applying Theorem \ref{thm:bounds_k_constraint_uncertainty} show that we need at most $k\in \mathcal O\left( (D+1)^p(n_\xi + 1)\right)$ second-stage policies. From Remark \ref{rem:objective_independent_of_uncertainty} it follows, that if the travel costs are not uncertain this bound improves to $k\in\mathcal O\left( (D+1)^p\right)$. In contrast, the number of second-stage solutions can be $|\mathcal Y|=p^{|\mathcal J|}$ since every customer in $\mathcal J$ can be assigned to at most $p$ facilities. Usually, the parameter $p$ is a small constant in facility location problems. If we now assume that $d_j=d$ for all $j$ then the latter bound on $k$ reduces to  $\mathcal O(|\mathcal J|^pd^p)$. For a fixed $p$ this bound grows polynomial in the number of customers while the number of second-stage solutions grows exponential in the number of customers. In Table \ref{tbl:values_facility_location} we compare both values for $p=5$ and $d=10$ for different number of customers.

\begin{table}[h!]
\centering
\begin{tabular}{|c|c|c|c|c|}
\hline
\textbf{$|\mathcal J|$} & \textbf{10} & \textbf{50} & \textbf{100} & \textbf{250} \\
\hline
$|\mathcal J|^pd^p$ & $ 10^{10}$ & $\approx 3.1\cdot 10^{13}$ & $ 10^{15}$ & $\approx 9.8\cdot 10^{16}$ \\
\hline
$p^{|\mathcal J|}$ & $\approx 9.8\cdot 10^6$ & $\approx 8.9\cdot 10^{34}$ & $\approx 7.9\cdot 10^{69}$ & $\approx 5.5\cdot 10^{174}$\\
\hline
\end{tabular}
\caption{Comparison of the bound on $k$ and the number of second-stage solutions for $p=5$ and $d=10$.}\label{tbl:values_facility_location}
\end{table}
}
\end{example}

\begin{remark}
The number of right-hand-side values in the $4$-step procedure can be improved in some situations. Note that for every right-hand-side value in \eqref{eq:polyhedral_description_feasible_constraints} which is larger than all left-hand-side values, the corresponding regions do not intersect with $\mathcal U$. On the other hand if the right-hand-side value is smaller than the smallest left-hand-side value the regions contain the full set $\mathcal U$. Hence, we do not need to consider these values in $\mathcal H(x)$.
% It follows that for 
% \[
% \gamma:=\max_{i=1,\ldots ,m}\max_{x\in\mathcal X}\max_{\xi\in \mathcal U} \ 2\left| \sum_{i=1}^{n_\xi} (A^ix + B^iy-H_i)_i\xi_i - h_i+ (Ax)_i\right|
% \]
% we can now bound $\eta \le m\gamma$.
\end{remark}

\rev{
Theorem \ref{thm:bounds_k_constraint_uncertainty} can be combined with Theorem \ref{thm:approximation_bound} to obtain approximation bounds for the constraint uncertainty case for certain values for $k$.
\begin{theorem}\label{thm:approximation_bound_constr}
Let $\mathcal U$ be convex and $g:\mathcal X\times \mathcal Y\times \mathcal U\to \mathbb R$ a continuous function such that $g(x, y, \xi)$ is concave in $\xi$ for every $x\in \mathcal X, y\in\mathcal Y$ and Lipschitz continuous in $y$, with Lipschitz constant $L>0$. Furthermore, assume that for every $x\in\mathcal X$ there exist $R$ convex recourse-stable regions $\mathcal D_1,\ldots ,\mathcal D_R\subseteq \mathcal U$ such that $\cl{\mathcal D_1}\cup\ldots\cup\cl{\mathcal D_R} = \mathcal U$. Then, for any $s\in \mathbb N$ with $s\le n_\xi + 1$ it holds
\[
\text{opt}(Rs)-\text{opt}(2RO) \le L \text{diam}(\mathcal Y) \rev{\ln\left(\frac{n_\xi+1}{s}\right)}.
\]
\end{theorem}
\begin{proof}
From Theorem \ref{thm:num_policies_generalcover} it follows that
\[
\text{opt}(2RO) = \text{opt}(R(n_\xi+1))
\]
and hence there exists an optimal solution $(\bar x,\bar y^{1},\ldots ,\bar y^{ R(n_\xi+1)})$ of the $R(n_\xi+1)$-adaptability problem which has the same optimal value as \eqref{eq:2StageRO_constr}. For ease of notation we denote $\bar y^{(t-1)(n_\xi +1) + i}=:\bar y^{t}_i$ for each $t\in [R]$. Following the proof of Theorem \ref{thm:bounds_k_constraint_uncertainty} we can w.l.o.g. assume that  $\bar y^{t}_1, \ldots ,\bar y^{t}_{n_\xi +1}$ are feasible for all $\xi\in \cl{\mathcal D_t}$ for every $t\in [R]$.

The idea of the proof is to remove all but $s$ second-stage solutions for each recourse stable region, leading to $Rs$ solutions. Since the maximum is attained over one of the recourse stable regions we can use the same bounding procedure as in Theorem \ref{thm:approximation_bound} which leads to the desired bound.

Similar as in the proof of Theorem \ref{thm:approximation_bound} we can reformulate for every $t\in [R]$
\[
\max_{\xi\in \cl{\mathcal D_t}}\min_{i=1,\ldots ,n_\xi+1} g(\bar x,\bar y^t_i,\xi) =   \max_{\xi\in \cl{\mathcal D_t}} \min_{\substack{\lambda \in \mathbb R_+^{n_\xi+1} \\ \sum_{i=1}^{n_\xi+1} \lambda_i = 1}} \ \sum_{i=1}^{n_\xi+1} \lambda_i g(\bar x,\bar y^t_i,\xi).
\]
Since $g$ is concave in $\xi$ and $\lambda \ge 0$, also the function $\sum_{i=1}^{n_\xi+1} \lambda_i g(x,y^i,\xi)$ is concave in $\xi$. We can apply the classical minimax theorem and swap the inner maximum and minimum operator which leads to the reformulation 
\[
    \min_{\substack{ \lambda \in \mathbb R_+^{n_\xi+1} \\ \sum_{i=1}^{n_\xi+1} \lambda_i = 1}} \max_{\xi\in \cl{\mathcal D_t}}\ \sum_{i=1}^{n_\xi+1} \lambda_i g(\bar x,\bar y^t_i,\xi).
\]
Let $\bar \lambda^{t}$ be an optimal solution of the latter problem and assume w.l.o.g. that $\bar \lambda_1^{t}\ge \ldots \ge \bar \lambda_{n_\xi+1}^{t}$. We define a feasible solution for the $Rs$-adaptability problem as
\[
x = \bar x, \quad y^{t}_{1} =\bar y^{t}_{1}, \ldots , y^{t}_{s}=\bar y^{t}_{s},
\]
and
\[
\lambda^t_1=\bar \lambda_1^{t}, \ldots , \lambda^t_{s-1} = \bar \lambda_{s-1}^{t}, \ \lambda^t_s = \sum_{i=s}^{n_\xi+1}\bar \lambda_i^{t}.
\]
Then, we have
\begin{align*}
\text{opt}(Rs)-\text{opt}(2RO) & \le \max_{t=1,\ldots ,R} \max_{\xi\in \cl{\mathcal D_t}}  \min_{\substack{\lambda \in \mathbb R_+^{s} \\ \sum_{i=1}^{s} \lambda_i = 1}} \ \sum_{i=1}^{s} \lambda_i g(\bar x,\bar y^t_i,\xi) \\
& \quad - \max_{t=1,\ldots ,R} \max_{\xi\in \cl{\mathcal D_t}}  \min_{\substack{\lambda \in \mathbb R_+^{n_\xi+1} \\ \sum_{i=1}^{n_\xi+1} \lambda_i = 1}} \ \sum_{i=1}^{n_\xi+1} \lambda_i g(\bar x,\bar y^t_i,\xi) \\
& = \max_{t=1,\ldots ,R} \min_{\substack{\lambda \in \mathbb R_+^{s} \\ \sum_{i=1}^{s} \lambda_i = 1}} \max_{\xi\in \cl{\mathcal D_t}} \ \sum_{i=1}^{s} \lambda_i g(\bar x,\bar y^t_i,\xi) \\
& \quad - \max_{t=1,\ldots ,R} \min_{\substack{\lambda \in \mathbb R_+^{n_\xi+1} \\ \sum_{i=1}^{n_\xi+1} \lambda_i = 1}} \max_{\xi\in \cl{\mathcal D_t}}   \ \sum_{i=1}^{n_\xi+1} \lambda_i g(\bar x,\bar y^t_i,\xi) 
\end{align*}
where the first inequality follows since the recourse stable regions form a cover of $\mathcal U$ and since we use the optimal solution for the $R(n_\xi+1)$-adaptable problem and a feasible solution for the $Rs$-adaptable problem. The equality follows from the minimax theorem as applied before. Let $t^*$ be the index, which maximizes the first maximum of the latter term. Then the latter term can be bounded from above by
\begin{align*}
 & \min_{\substack{\lambda \in \mathbb R_+^{s} \\ \sum_{i=1}^{s} \lambda_i = 1}}\max_{\xi\in \cl{\mathcal D_{t^*}}}\ \sum_{i=1}^{s} \lambda_i g(\bar x,\bar y^{t^*}_i,\xi) - \min_{\substack{\lambda \in \mathbb R_+^{n_\xi+1} \\ \sum_{i=1}^{n_\xi+1} \lambda_i = 1}}\max_{\xi\in \cl{\mathcal D_{t^*}}} \ \sum_{i=1}^{n_\xi+1} \lambda_i g(\bar x,\bar y^{t^*}_i,\xi) \\
& \le \max_{\xi\in \cl{\mathcal D_{t^*}}}\ \sum_{i=1}^{s} \lambda^{t^*}_i g(\bar x,\bar y^{t^*}_i,\xi)
- \max_{\xi\in \cl{\mathcal D_{t^*}}} \ \sum_{i=1}^{n_\xi+1} \bar \lambda_i^{t^*} g(\bar x,\bar y^{t^*}_i,\xi),
\end{align*}
where the inequality follows since by definition the $\lambda_i^{t^*}$ are feasible for the minimum operator of the first min-max problem and $\bar \lambda_i^{t^*}$ are optimal for the second min-max problem. Let $\xi^*$ be a scenario which maximizes the first maximum of the latter expression. Then we can further bound the latter term from above by
\begin{align*}
    & \le \sum_{i=1}^{s} \lambda^{t^*}_i g(\bar x,\bar y^{t^*}_i,\xi^*)
-  \sum_{i=1}^{n_\xi+1} \bar \lambda_i^{t^*} g(\bar x,\bar y^{t^*}_i,\xi^*) \\
    & = \sum_{i=s+1}^{n_\xi + 1} \bar \lambda_i^{t^*} \left( g(\bar x,\bar y^{t^*}_s,\xi^*) - g(x^*,\bar y^{t^*}_i,\xi^*) \right) \\
    & \le L \sum_{i=s+1}^{n_\xi  + 1}\bar \lambda_i^{t^*}\| \bar y^{t^*}_s - \bar y^{t^*}_i \| \\
    & \le L \text{diam}(\mathcal Y) \sum_{i=s+1}^{n_\xi + 1}\bar \lambda_i^{t^*},
\end{align*}
where the first inequality follows since $\xi^*$ is optimal for the first maximum and feasible for the second maximum, the first equality follows from the definition of  $\lambda_s^{t^*}$, the second inequality follows from the Lipschitz continuity of $g$, and the last inequality follows from the definition of the diameter. From the sorting $\bar \lambda_1^{t^*}\ge \ldots \ge \bar \lambda_{n_\xi +1}^{t^*}$ and since $\sum_{i=1}^{n_\xi + 1}\bar \lambda_i^{t^*} = 1$ it follows $\bar \lambda_i^{t^*}\le \frac{1}{i}$. Hence we can further bound 
\begin{align*}
    L \text{diam}(\mathcal Y) \sum_{i=s+1}^{n_\xi +1}\bar \lambda_i^{t^*} \le L \text{diam}(\mathcal Y)\sum_{i=s+1}^{n_\xi +1}\frac{1}{i}
    \rev{\le L \text{diam}(\mathcal Y) \int_{s}^{n_\xi +1} \frac{1}{t}dt = L \text{diam}(\mathcal Y)\ln\left(\frac{n_\xi +1}{s}\right),}
\end{align*}
where the last inequality and equality follow from \rev{classical theory of integrals}. This proves the result.
\end{proof}
The latter theorem shows approximation bounds for a limited number of $k$-values, namely for the cases $k=Rs$ with $1\le s\le n_\xi +1$. However, the bound for $k=Rs$ holds for all $k\in\{Rs+1, Rs+2, \ldots, Rs + R-1\}$ as well, due to the decreasing optimal value for larger $k$, although it may be less precise.
}
\section{Finite Uncertainty Sets}\label{sec:finite_uncertainty}
In this section we study the case where the uncertainty set is finite, i.e, $\mathcal U = \{ \xi^1, \ldots, \xi^t\}$. We compare the more general problems \eqref{eq:2StageRO_constr} and \eqref{eq:k-adaptability_constr} where the uncertainty can appear in the constraints. Note that finite uncertainty sets naturally appear in data-driven applications where usually a finite set of historical scenarios or a finite set of samples is available for the uncertain parameters. Furthermore, by sampling scenarios from a convex uncertainty set, the finite uncertainty case can be interpreted as an approximation of the convex case.

\subsection{Complexity}
We will show in this section, that \eqref{eq:k-adaptability_constr} with finite uncertainty sets is closely connected to the set covering problem. We can exploit this insight to (i) show that calculating the minimal $k$ for which \eqref{eq:2StageRO_constr} and \eqref{eq:k-adaptability_constr} are equivalent is NP-hard, and (ii) derive a greedy algorithm which provides an approximation of the minimal $k$ value if there are no first-stage decisions. We first define the \textit{$k$-equivalence problem} and the \textit{set cover problem}.

\begin{definition}
For a given $k\in\mathbb N$ and an instance of \eqref{eq:2StageRO_constr}, the \textit{$k$-equivalence problem} answers the question if  the optimal value of \eqref{eq:k-adaptability_constr} is equal to the optimal value of \eqref{eq:2StageRO_constr}. 
\end{definition}
Note that such a $k\in \mathbb N$ always exists \rev{if \eqref{eq:2StageRO_constr} has at least one feasible solution $x\in\mathcal X$}, since $k=|\mathcal Y|$ is a valid choice. A closely connected problem is the following:
\begin{definition}
For a given number $\kappa\in \mathbb N$, a finite universe $\mathcal V$, and a finite set $\mathscr{S}$ of subsets of $\mathcal V$ \rev{with $\bigcup_{\mathcal S\in \mathscr S} S = \mathcal V$}, the set cover problem answers the question if there exist $\kappa$ sets $\mathcal S_1, \ldots ,\mathcal S_\kappa\in \mathscr S$ such that $\mathcal V = \bigcup_{i=1}^{\kappa} \mathcal S_i$. 
\end{definition}
\rev{Note that the assumption $\bigcup_{\mathcal S\in \mathscr S} S = \mathcal V$ is not restrictive since for any given sets  $\mathcal V$ and $\mathscr{S}$ we can always check in time polynomial in the input size if the sets in $\mathscr{S}$ indeed cover $\mathcal V$ and return the answer no otherwise.}

The main intuition why the latter two problems are closely related is the following: Assume we know the optimal value of \eqref{eq:2StageRO_constr}, denoted as $v^*$ and an optimal solution $x^*\in \mathcal X$. Consider the universe $\mathcal V = \mathcal U$ and define the set 
\[
\mathscr{S}:=\left\{ \mathcal S_y: y\in \mathcal Y(x^*) \right\} 
\]
where
\[
\mathcal S_y=\{ \xi\in \mathcal U: A(\xi) x^* + B(\xi)y \ge h(\xi), \ g(x,y,\xi) \le v^* \},
\]
i.e., the set $\mathcal S_y$ contains all scenarios from $\mathcal U$, for which the corresponding second-stage solution $y$ is feasible and the objective value is at most $v^*$. Since $\mathcal U$ is finite, all sets $\mathcal S_y$ are finite. If we can find a set cover $\mathcal S_{y^1},\ldots ,\mathcal S_{y^\kappa}\in \mathscr S$ for $\mathcal U$, then $(x^*,y^1,\ldots ,y^\kappa)$ is a feasible solution for \eqref{eq:k-adaptability_constr} with objective value $v^*$ and hence for $k=\kappa$ Problem \eqref{eq:2StageRO_constr} and \eqref{eq:k-adaptability_constr} have the same optimal value. 
% It can be easily verified that no smaller value for $k$ exists which has the same property, since otherwise the corresponding second-stage solutions would form a cover of $\mathcal U$ which is a contradiction to the minimality of the set cover solution. 

We can use the latter intuition to prove the following theorem. 
\begin{theorem}\label{thm:np-completeness}
    The $k$-equivalence problem is NP-complete, even if the uncertainty only appears in the objective function.
\end{theorem}
\begin{proof}
We reduce the NP-complete set cover problem to the $k$-equivalence problem. Consider an arbitrary instance of the set cover problem with universe $\mathcal V$ and finite set of subsets $\mathscr S$ and $\kappa\in \mathbb N$. Define an instance of \eqref{eq:2StageRO_constr} as follows: Define the first-stage feasible set $\mathcal X = \{ 0\}$,  the dimensions $n_y = n_\xi = |\mathcal V|$, the uncertainty set $\mathcal U=\{ \xi^v: v\in \mathcal V\}$, where $\xi^v = e_v$ is the unit vector with entry $1$ for element $v$. The second-stage feasible set is given as $\mathcal Y=\{ y^\mathcal S: \mathcal S\in \mathscr S\}$, where $y^{\mathcal S}_v = 1$ if $v\in \mathcal S$ and $y^{\mathcal S}_v = 0$ otherwise. Furthermore, $g(x,y,\xi) = -\xi^\top y$ and there are no further constraints in the second-stage, hence the uncertainty appears only in the objective function. For any cover $\mathcal S_1, \ldots ,\mathcal S_{\kappa'}$ of $\mathcal V$ we have that for every $\xi\in U$ there exists an $i\in [\kappa']$ such that $g(x,y^{\mathcal S_i},\xi) = -1$, i.e. the optimal value of \eqref{eq:k-adaptability_constr} with $k=\kappa'$ is $-1$. On the other hand, if for a given $\kappa$ no cover exists, then \eqref{eq:k-adaptability_constr} with $k=\kappa$ must have optimal value $0$. Since \rev{by assumption} there exists a cover of unknown size which covers the set $\mathcal V$ the \eqref{eq:2StageRO_constr} must have optimal value $-1$. Hence \eqref{eq:2StageRO_constr} and \eqref{eq:k-adaptability_constr} have the same optimal value for $k=\kappa$ if and only if the answer to the set cover problem is yes, which proves the result.
\end{proof}

\subsection{Calculating $k$ for Min-max-min Robust Optimization}
We will now derive a greedy algorithm that provides an approximation for the minimal $k$ for which the two-stage robust problem and the $k$-adaptability problem have the same optimal value in case there are no first-stage decisions. In this case \eqref{eq:k-adaptability_constr} reduces to
\begin{equation}\label{eq:min-max-min}
\min_{y^1,\ldots ,y^k\in \mathcal Y} \max_{\xi\in\mathcal U} \min_{\substack{i\in [k]: \\ B(\xi)y^i\ge h(\xi)}} g(y^i,\xi)
\end{equation}
and the two-stage robust problem \eqref{eq:2StageRO_constr} reduces to \eqref{eq:min-max-min} with $k=|\mathcal Y|$. Note that the optimal value of \eqref{eq:2StageRO_constr} in this case is
\[
v^*:=\max_{\xi\in\mathcal U} \min_{y\in\mathcal Y} g(y,\xi)
\]
which can be calculated by solving the deterministic problem $\min_{y\in\mathcal Y} g(y,\xi)$ for each of the $t$ scenarios in $\mathcal U$. Then calculating the minimum value of $k$ for which \eqref{eq:min-max-min} has optimal value $v^*$ can be done by calculating the minimum value $\kappa$ for which a set cover of the following set-cover instance exists: define $\mathcal V = \mathcal U$, 
\[
\mathscr{S}:=\left\{ \mathcal S_y: y\in \mathcal Y \right\} 
\]
and
\[
\mathcal S_y:=\{ \xi \in \mathcal U: B(\xi)y\ge h(\xi), \ g(y,\xi)\le v^*\} .
\]
The latter problem is also called \textit{minimum set cover problem}. While this problem is computationally very challenging to solve there exist fast approximations algorithms, e.g. the classical greedy algorithm \cite{williamson2011design}. The idea of the algorithm is to iteratively select the set $\mathcal S_y$ which covers most of the elements in $\mathcal U$ which are still uncovered by the previously selected sets. For a given set of uncovered scenarios $\mathcal U'\subset \mathcal U$, the set $\mathcal S_y$ which covers the maximum number of elements in $\mathcal U'$ can be calculated by solving the following integer optimization problem:
\begin{equation}\label{eq:max_set_cover}
\begin{aligned}
    \rev{\text{opt}(\mathcal U'):}= \max_{y,z} \ & \sum_{\xi\in \mathcal U'} z_\xi \\
    s.t. \quad & g(y,\xi) \le v^* + M(1-z_\xi) \quad \forall \xi\in\mathcal U' \\
    & B(\xi)y\ge h(\xi) - M(1-z_\xi) \quad \forall \xi\in\mathcal U' \\
    & y\in\mathcal Y, z\in\{ 0,1\}^{\mathcal U'}
\end{aligned}
\end{equation}
where $M$ is a large enough value. The binary variable $z_\xi$ can only be set to one if the selected solution $y$ is feasible for $\xi$ and has objective value at most $v^*$ which is ensured by the big-M constraints. Hence, the optimal value of the latter problem is equal to the maximum number of scenarios in $\mathcal U'$ which can be covered by a set $\mathcal S_y$. The whole greedy algorithm is presented in Algorithm \ref{alg:greedy_setcover}. Note that while the algorithm returns the number of solutions needed, at the same time it also calculates a set of feasible solutions.

\begin{algorithm}
\caption{Greedy Algorithm to Approximate the Minimal $k$}\label{alg:greedy_setcover}
\begin{algorithmic}
\Require $\mathcal U$, $\mathcal Y$, $B(\xi)$, $h(\xi)$, $g$
\Ensure $k_{\text{ub}}$
\State Set of solutions $\mathcal Y_{\text{sol}} = \emptyset$
\State Set of uncovered scenarios $\mathcal U' = \mathcal U$
\While{$\mathcal U'\neq \emptyset$}
\State Calculate optimal solution $y_{\mathcal U'}^*,z_{\mathcal U'}^*$ of \eqref{eq:max_set_cover}.
\State $\mathcal Y_{\text{sol}}\leftarrow \mathcal Y_{\text{sol}} \cup \{ y_{\mathcal U'}^*\}$
\State Remove all scenarios $\xi$ from $\mathcal U'$ for which $z_{\mathcal U'}^*=1$.
\EndWhile \\
\Return $k_{\text{ub}}:=|\mathcal Y_{\text{sol}}|$
\end{algorithmic}
\end{algorithm}

\begin{theorem}
    Algorithm \ref{alg:greedy_setcover} calculates a value $k_{\text{ub}}$ with 
    \[
    k_{\text{ub}} \le \left( 1+\ln \left( t \right)\right) k_{\text{opt}}
    \]
    where $k_{\text{opt}}$ is the minimum value for $k$ for which \eqref{eq:2StageRO_constr} and \eqref{eq:k-adaptability_constr} have the same optimal value.
\end{theorem}
\begin{proof}
    The results directly follows from the approximation guarantee of the greedy algorithm for the minimum set cover problem; see for example \cite{williamson2011design}.
\end{proof}

\begin{remark}
Algorithm \ref{alg:greedy_setcover} also provides a lower bound on the value $k_{\text{opt}}$. To see this, note that the number of scenarios which is covered by $y^*$ in the first iteration of the while-loop \rev{(i.e. $\text{opt}(\mathcal U)$)} is the maximum number of scenarios which can be covered by any solution $y\in \mathcal Y$. Hence, a lower bound on $k_{\text{opt}}$ is given as
\[
k_{\text{lb}} = \Bigl\lceil \frac{t}{\rev{\text{opt}(\mathcal U)}} \Bigr\rceil \le k_{\text{opt}}.
\]
\end{remark}

\paragraph{Experiments} The following experiment shows that indeed Algorithm \ref{alg:greedy_setcover} is able to calculate reasonable approximations for the minimum $k$ value in reasonable time. Unfortunately, there exists no method which is able to calculate the optimal value for $k$ \rev{in reasonable time} even for small dimensions \rev{$n_x,n_y$ and $n_\xi$}. Hence, no comparison with the optimal value can be made. However, the values calculated by the greedy algorithm are significantly smaller than the number of scenarios $t$ and the lower-bound provides an estimation of the quality of the calculated value.

We consider the $k$-adaptable version of the minimum knapsack problem which is defined as
\[
\min_{y^1,\ldots ,y^k\in \mathcal Y} \max_{\xi\in\mathcal U} \min_{\substack{i\in [k]: \\ a(\xi)^\top y^i\ge b}} c(\xi)^\top y^i
\]
where objective parameters $c(\xi)$ and the constraint parameters $a(\xi)$ are uncertain. We generate instances of the problem as follows. For a given dimension $n_y$ we draw a random vector $\bar a \in \mathbb N^{n_y}$ where each component $\bar a_i$ is drawn uniformly at random from $\{ 40,41,\ldots ,60\}$.  The knapsack capacity is set to $b=0.2\sum_{i=1}^{n_y} \bar a_i$. We draw a random vector $\bar c$ where each component $\bar c_i$ is drawn uniformly at random from $\{ \bar a_i -5, \bar a_i-4, \ldots ,\bar a_i + 5\}$. Then we define a deviation parameter $\Gamma=\lfloor \frac{n_y}{4}\rfloor$ and generate $m$ scenarios $\{(c(\xi^j),a(\xi^j))\}_{j\in [m]}$ as follows: for each scenario we pick $\Gamma$ random indices from $[n_y]$ and set $c(\xi^j)_i = 1.5\bar c_i$ for each of the latter indices. For all other indices we set $c(\xi^j)_i = \bar c_i$. Such a scenario can be interpreted as a random scenario from a budgeted uncertainty set where at most $\Gamma$ many value can deviate from its mean value. We follow the same procedure to generate the vector $a(\xi^j))$.

We present the runtime and the values for $k_{\text{ub}}$ and $k_{\text{lb}}$ calculated by Algorithm \ref{alg:greedy_setcover} in Figure \ref{fig:greedy_alg_over_t} and \ref{fig:greedy_alg_over_n}. In Figure \ref{fig:greedy_alg_over_t} we show the latter values for varying number of scenarios $t$. In Figure \ref{fig:greedy_alg_over_n} we show the same values for varying dimension $n_y$. The results indicate that the calculated upper bound $k_{\text{ub}}$ is close to the lower bound $k_{\text{lb}}$ and hence close to the optimal $k$ \rev{compared to the trivial bound $\bar k_{ub} = t$ (optimality gap of $200\%$ to $400\%$ for varying $t$ and $100\%$ to $250\%$ for varying $n_y$)}. Furthermore it can be seen that the optimal $k$ is quite small compared to the number of scenarios $t$. The values $k_{\text{ub}}$ and $k_{\text{lb}}$ increase with increasing dimension $n_y$ and number of scenarios. Similarly, the runtime increases with increasing dimension $n_y$ and number of scenarios and is never larger than $800$ seconds.

\begin{figure}
    \centering
    \includegraphics[width=0.45\linewidth]{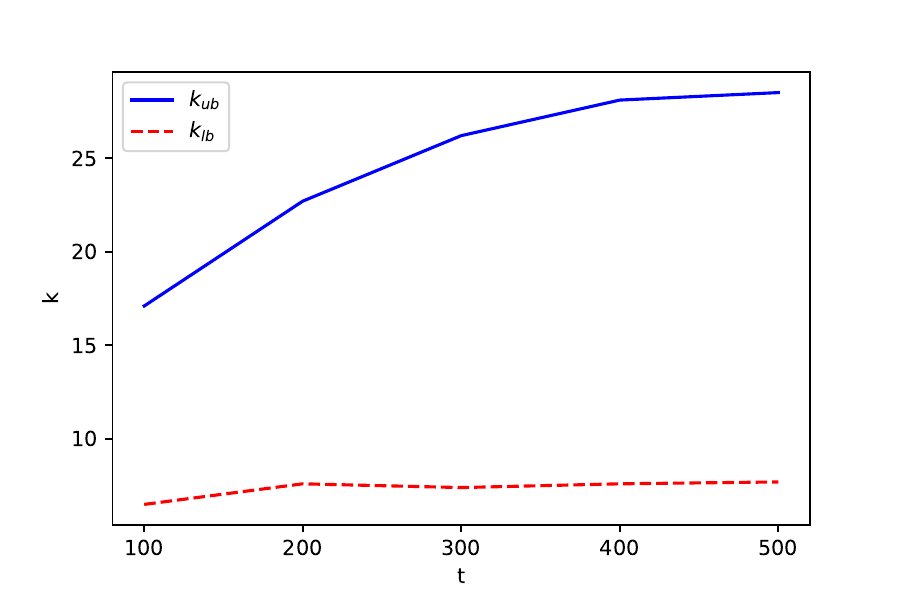} \quad 
    \includegraphics[width=0.45\linewidth]{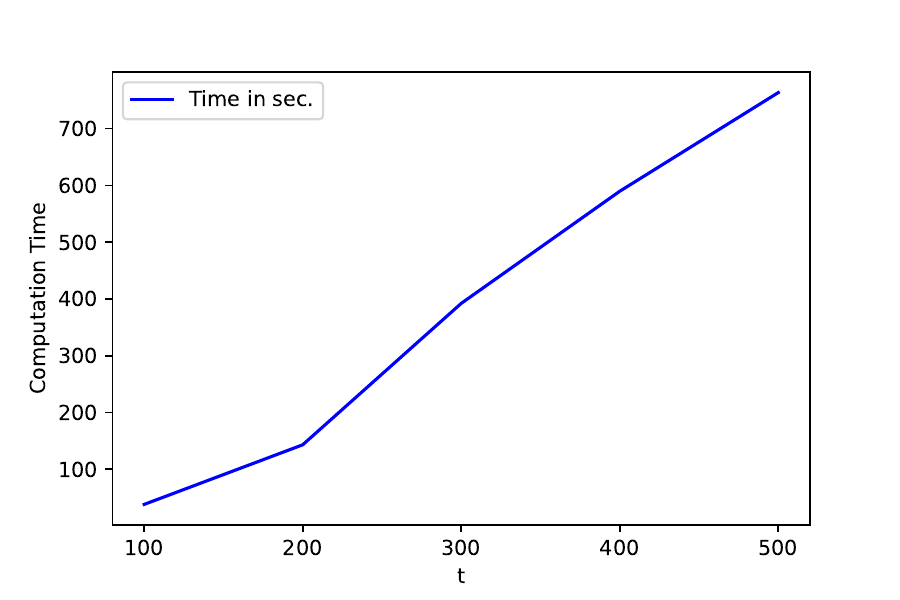}
    \caption{Lower bound $k_{\text{lb}}$ and upper bound $k_{\text{ub}}$ on optimal $k$ value calculated by Algorithm \ref{alg:greedy_setcover} (left) and runtime of Algorithm \ref{alg:greedy_setcover} (right) over the number of scenarios $t$ with dimension $n_y=20$.}
    \label{fig:greedy_alg_over_t}
\end{figure}

\begin{figure}
    \centering
    \includegraphics[width=0.45\linewidth]{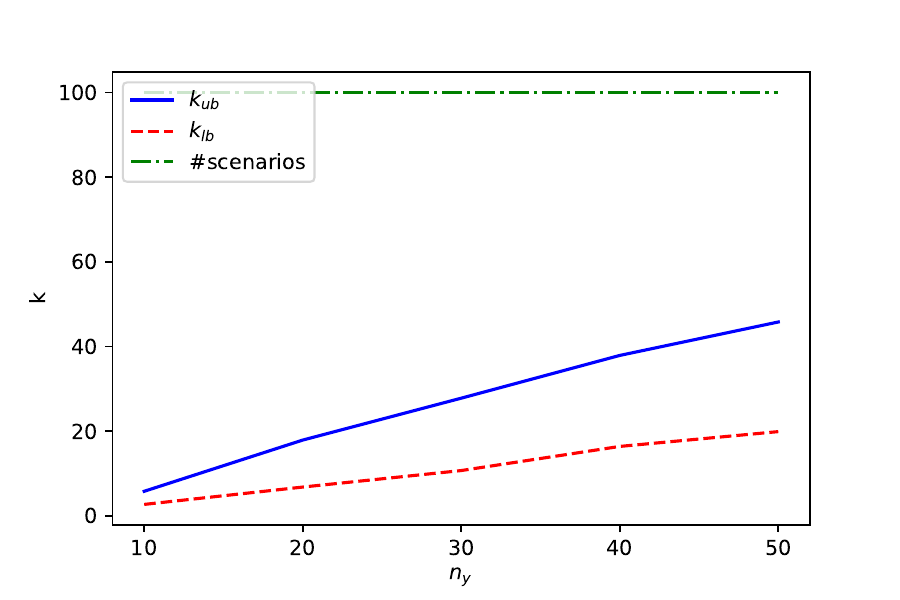} \quad 
    \includegraphics[width=0.45\linewidth]{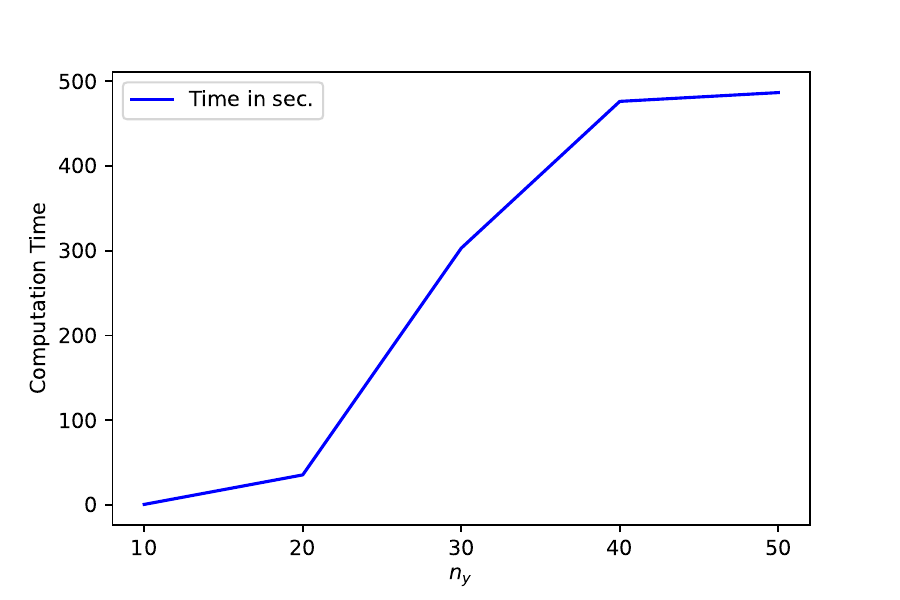}
    \caption{Lower bound $k_{\text{lb}}$ and upper bound $k_{\text{ub}}$ on optimal $k$ value calculated by Algorithm \ref{alg:greedy_setcover} (left) and runtime of Algorithm \ref{alg:greedy_setcover} (right) over the dimension $n_y$ with $t=100$ scenarios.}
    \label{fig:greedy_alg_over_n}
\end{figure}

\section{Conclusion}
In this work we derived bounds on the number $k$ of second-stage solutions which are needed such that the $k$-adaptability approach returns an optimal solution for the original two-stage robust problem \rev{or an approximately optimal solution}. We distinguished the two cases of objective uncertainty and constraint uncertainty for convex uncertainty sets. Interestingly, for objective uncertainty the number of solutions needed is $k=n_\xi +1$, i.e., it depends only on the dimension of the uncertainty. This results hold for a very general class of continuous objective functions. We used the latter result to derive approximation guarantees the $k$-adaptability problem provides for all values of $k$ smaller than the bound above.

For convex constraint uncertainty we developed a new concept called recourse-stability. A recourse-stable region is a subset of the uncertainty set such that each second-stage solution is either feasible or infeasible on the whole set. We could show that to guarantee optimality we need at most $k=R(n_\xi + 1)$ second-stage solutions, where $R$ is the number of recourse-stable regions needed to cover the uncertainty set. We show that we can determine a value for $R$ by considering all possible hyperplanes which can appear in the uncertain constraints by plugging in all possible second-stage solutions. Examples show that the derived value for $R$ provides good bounds on $k$ for many problem structures. 

Finally, for finite uncertainty sets we present a connection between $k$-adaptability and the set cover problem which can be used to show that the problem of finding the minimum $k$ value is NP-hard. Furthermore, this connection leads to a greedy algorithm which approximates the optimal $k$ if there are no first-stage decisions.

There remain several open questions to be tackled in the future. First, it would be interesting if we can achieve better values for $R$ by focusing on certain applications and the corresponding problem structures. Furthermore, it would be interesting if \rev{more general} approximation bounds as for objective uncertainty could also be derived for the constraint uncertainty case. Finally, an important question is if the methodology derived in this work can contribute to the development of efficient solution methods.

\paragraph*{Acknowledgement}
The author wants to thank Henri Lefebvre and Dick den Hertog for the fruitful discussions about the problem and consequent insights.

\paragraph{Funding and Competing Interests}
The author did not receive support from any organization for the submitted work. The author has no relevant financial or non-financial interests to disclose.

\bibliographystyle{alpha}
\bibliography{references}

\end{document}